\documentclass[11pt,reqno]{amsart}
\usepackage{amsfonts, amssymb, amsmath, dsfont}
\usepackage[linktoc=page, colorlinks, linkcolor=blue, citecolor=blue]{hyperref}
\usepackage{enumerate}
\numberwithin{equation}{section}

\DeclareMathOperator{\E}{\mathbb{E}}

\DeclareMathOperator{\Ber}{Ber}
\DeclareMathOperator{\Var}{Var}
\DeclareMathOperator{\Binom}{Binom}

\renewcommand{\Pr}[2][]{\mathbb{P}_{#1} \left\{ #2 \rule{0mm}{3mm}\right\}}
\newcommand{\ip}[2]{\langle#1,#2\rangle}

\def \P {\mathbb{P}}
\def \R {\mathbb{R}}

\def \EE {\mathcal{E}}
\def \GG {\mathcal{G}}
\def \FF {\mathcal{F}}

\def \PP {\mathcal{P}}
\def \RR {\mathcal{R}}
\def \BB {\mathcal{B}}

\def \e {\varepsilon}
\def \d {\delta}

\def \s {\sigma}

\def \tran {\mathsf{T}}

\def \one {{\textbf 1}}
\def \ind {{\mathds 1}}

\newtheorem{theorem}{Theorem}[section]
\newtheorem{proposition}[theorem]{Proposition}
\newtheorem{corollary}[theorem]{Corollary}
\newtheorem{lemma}[theorem]{Lemma}

\theoremstyle{remark}
\newtheorem{remark}[theorem]{Remark}

\begin{document}

\title{Norms of random matrices: local and global problems}
\author{Elizaveta Rebrova \and Roman Vershynin}

\address{Department of Mathematics, University of Michigan, 530 Church St, Ann Arbor, MI 48109, U.S.A.}
\email{erebrova@umich.edu, romanv@umich.edu}

\thanks{R. V. is partially supported by NSF grant 1265782 and U.S. Air Force grant FA9550-14-1-0009.}

\begin{abstract}
Can the behavior of a random matrix be improved by modifying a small fraction of its entries?
Consider a random matrix $A$ with i.i.d. entries.
We show that the operator norm of $A$ can be reduced to the optimal order $O(\sqrt{n})$ 
by zeroing out a small submatrix of $A$ if and only if the entries have zero mean and finite variance. 
Moreover, we obtain an almost optimal dependence 
between the size of the removed submatrix and the resulting operator norm.
Our approach utilizes the cut norm and Grothendieck-Pietsch factorization for matrices, and it combines 
the methods developed recently by C.~Le and R.~Vershynin and by E.~Rebrova and K.~Tikhomirov.
\end{abstract}

\maketitle

\section{Introduction}

\subsection{Local and global problems}		\label{s: local global}
When a certain mathematical or scientific structure fails to meet reasonable expectations, 
one often wonders: is this a local or global problem? 
In other words, is the failure caused by some small, localized part of the structure, and if so,
can this part be identified and repaired? Or, alternatively, is the structure entirely, globally bad?
Many results in mathematics can be understood as either local or global statements. 
For example, not every measurable function $f : \R \to \R$ is continuous, but 
Lusin's theorem implies that $f$ can always be made continuous by changing its values on a set 
of arbitrarily small measure. Thus, imposing continuity is a local problem. 
On the other hand, a continuous function may not be differentiable, and there
even exist continuous and nowhere differentiable functions. Thus imposing differentiability may 
be a global problem. In statistics, the notion of {\em outliers} -- 
small, pathological subsets of data, the removal of which makes data better --
points to local problems.

\subsection{Random matrices and their norms}
In this paper we ask: is bounding the norm of a random matrix a local or a global problem? 
To be specific, consider $n \times n$ random matrices $A$
with independent and identically distributed (i.i.d.) entries. The {\em operator norm} of $A$ is 
defined by considering $A$ as a linear operator on $\R^n$ equipped with the Euclidean norm $\|\cdot\|_2$, i.e. 
$$
\|A\| = \max_{x \ne 0} \frac{\|Ax\|_2}{\|x\|_2}.
$$

Suppose that the entries of $A$ have zero mean and bounded fourth moment, 
i.e. $\E A_{ij}^4 = O(1)$. Then, as it was shown in \cite{Bai-Yin-Kri},
$$
\|A\| = (2+o(1)) \sqrt{n}
$$
with high probability. Note that the order $\sqrt{n}$ is the best we can generally hope for. Indeed, if the entries of $A$ have unit variance, then the typical magnitude of the Euclidean norm of a row of $A$ is $\sim \sqrt{n}$, 
and the operator norm of $A$ can not be smaller than that.  
Moreover, by \cite{BSY, Silv} the bounded fourth moment assumption is nearly necessary\footnote{For almost surely convergence of $\|A\|/\sqrt{n}$, fourth moment is necessary and sufficient \cite{BSY}, while for convergence in probability the weak fourth moment is necessary and sufficient~\cite{Silv}.}
for the bound
\begin{equation}         \label{eq: sqrtn bound}
\|A\| = O(\sqrt{n}).
\end{equation}
%
%
%
%
%
A number of quantitative and more general versions of these bounds are known \cite{Seginer, Latala, Vu, BvH, vH, vH structured}.

\subsection{Main results}

Now let us postulate nothing at all about the distribution of the i.i.d. entries of $A$. It still makes sense to ask: 
{\em is enforcing the ideal bound \eqref{eq: sqrtn bound} for random matrices a local or a global problem}?
That is, can we enforce the bound \eqref{eq: sqrtn bound} by modifying the entries in a small submatrix of $A$?
We will show in this paper that this is possible if and only if the entries of $A$ have zero moment and finite variance.
The ``if'' part is covered by the following theorem.

\begin{theorem}[Local problem]  	\label{main}
  Consider an $n \times n$ random matrix $A$ with i.i.d. entries that have zero mean and unit variance,
  and let $\e \in (0,1/6]$. Then, with probability at least $1 - 7\exp(- \e n/12)$, 
  there exists an $\e n \times \e n$ submatrix of $A$ such that replacing all of its entries with zero 
  leads to a well-bounded matrix $\tilde{A}$: 
  $$
  \|\tilde{A}\| \leq \frac{C \ln \e^{-1}}{\sqrt{\e}} \cdot \sqrt{n},
  $$
    where $C$ is a sufficiently large absolute constant. 
\end{theorem}

\begin{remark}[Optimality]			\label{rem: optimality}
  The dependence on $\e$ in Theorem~\ref{main} is best possible up to the $\ln \e^{-1}$ factor. 
  To see this, let $p := 2\e/n$ and suppose $A_{ij}$ take values $\pm 1/\sqrt{p}$ with probability $p/2$ each
  and value $0$ with probability $1-p$. Then $A_{ij}$ have zero mean and unit variance as required.
  The expected number of non-zero entries in $A$ equals $pn^2 = 2\e n$.
  Thus the number of the rows of $A$ containing these entries is bigger than $\e n$ with high probability. 
  (This is a standard observation about the balls-into-bins model.) 
  Therefore, no $\e n \times \e n$ submatrix can contain all the non-zero entries of $A$. 
  In other words, $\tilde{A}$ must contain at least one non-zero entry of $A$, and thus it has magnitude
  $$
  \|\tilde{A}\| \ge \frac{1}{\sqrt{p}} \gtrsim \frac{\sqrt{n}}{\sqrt{\e}}.
  $$
  This shows that the dependence on $\e$ in Theorem~\ref{main} is almost optimal.
\end{remark}

By rescaling, a more general version of Theorem~\ref{main} holds for any finite variance of the entries. 
The two main assumptions in this theorem -- mean zero and finite variance -- are necessary in Theorem~\ref{main}.
Without either of them, the problem becomes global in a strong sense: the desired $O(\sqrt{n})$ bound
can not be achieved even after modifying a {\em large} submatrix. This is the content 
of the following result. 

\begin{theorem}[Global problem]\label{global}
  Consider an $n \times n$ random matrix $A_n$ whose entries are i.i.d. copies of a
  random variable that has either nonzero mean or infinite second moment,\footnote{Although this is a minor 
    terminological distinction, in this theorem we prefer to talk about second moment rather than variance. 
    This is because the second moment $\E X^2$ of a random variable $X$ 
    is always defined in the extended real line, while the variance $\Var(X) = \E(X-\E X)^2$ is undefined 
    if the mean $\E X$ is infinite.}
  and let $\e \in (0,1)$. Then 
  $$
  \min \frac{\|\tilde{A}_n\|}{\sqrt{n}} \to \infty \quad \text{as} \quad n \to \infty
  $$
  almost surely. Here the minimum is with respect to the matrices $\tilde{A}_n$ 
  obtained by any modification of any $\e n \times \e n$ submatrix of $A_n$.
\end{theorem}

It should be noted that while Theorem~\ref{main} becomes harder for smaller $\e$, 
Theorem~\ref{global} becomes harder for larger $\e$, those near $1$. 

We prove Theorem~\ref{global} in Section \ref{s: global}. The argument is considerably simpler than 
for Theorem~\ref{main}. Indeed, the nonzero mean forces the sum of the entries 
of $\tilde{A}_n$ to be $\gtrsim n^2$, and the infinite second moment forces the Frobenius norm of $\tilde{A}_n$
(the square root of the sum of the entries squared) to be $\gg n^2$ with high probability. 
Either of these two bounds can be easily used to show that the operator norm of $\tilde{A}_n$ is $\gg \sqrt{n}$.

\subsection{What if we remove large entries?}			\label{s: 2plusepsilon}

One may naturally wonder what exactly may cause the norm of a mean zero random matrix $A$ 
to be too large. A natural guess is that the only troublemakers are a few large entries of $A$. 
Indeed, this is exactly how the necessity of the fourth moment for \eqref{eq: sqrtn bound} was shown in \cite{BSY, Silv}. So we may ask -- can we obtain a result like Theorem~\ref{main} simply by zeroing out 
a few largest entries of $A$?

The answer is no. A counterexample is a sparse Bernoulli matrix $A$, 
whose i.i.d. entries take values $\pm \sqrt{n}$ with probability $1/2n$ each and $0$ with probability $1-1/2n$.
It is not hard to check that $A$ is likely to have a row whose norm exceeds 
$c\sqrt{n} \log (n) / \log\log n \gg \sqrt{n}$, and consequently we have $\|A\| \gg \sqrt{n}$.
In other words, without removal of any entries the norm of $A$ is too large.
However, if we are to remove any entries based purely on their magnitudes, we must remove them all.
(Recall that all non-zero elements of $A$ have the same magnitude $\sqrt{n}$.)
But removal of all nonzero entries of $A$ is not a local intervention, since such entries can not be placed in 
a small submatrix (we explained this in Remark~\ref{rem: optimality}). 

Nevertheless, under slightly stronger moment assumptions than in Theorem~\ref{main}, 
zeroing out a few large entires does bring the norm of $A$ down. The following result can be quickly deduced
by truncation from known bounds on random matrices such as \cite{vH, Seginer, Auff}. 

\begin{proposition}[$2+\e$ moments]	\label{twoplus} 
  For any $\e \in (0, 1]$ there exists $n_0 = n_0(\e)$ such that the following holds for any $n > n_0(\e)$.
  Consider an $n \times n$ random matrix $A$ with i.i.d. mean zero entries which satisfy
  $\E |A_{ij}|^{2 + \e} \le 1$. 
  Then, with the probability at least $1 - 2\exp(-n^{\e/5})$, there exists a integer $K \le n^{1 - \e/9}$ 
  such that the matrix $\tilde{A}$ obtained by zeroing out $K$ largest entries of $A$ satisfies
  $$
  \|\tilde{A}\| \le 9\sqrt{n}.
  $$
\end{proposition}

We will deduce Proposition~\ref{twoplus} from a general bound 
of A.~Bandeira and R.~van~Handel \cite{BvH} in Section~\ref{s: twopluseps}.

\subsection{Other related results}			\label{s: related}

For Bernoulli random variables, a variant of Theorem~\ref{main} was proved by 
U.~Feige and E.~Ofek \cite{FO}; see \cite{LV} for an alternative argument and more general way 
to regularize such matrices. Suppose the entries of an $n \times n$ matrix $B$ are independent Bernoulli random variables with mean $p \in (0,1)$.
If one removes the heavy rows and columns -- those containing more than $2pn$ ones, then the resulting matrix $B'$ satisfies 
the optimal norm bound $\|B' - \E B'\| = O(\sqrt{pn})$. To see that this bound is consistent with that of Theorem~\ref{main}, 
divide both sides by $\sqrt{p}$ to normalize the variance of the entries. 
Moreover, one can quickly check using concentration 
that the number of heavy rows and columns in $B$ is typically small. With a little more work, one can even place all ones
from the heavy rows and columns into a small submatrix (see Lemma~\ref{lem: Bernoulli} below). 
Thus Feige-Ofek's result is an example of Theorem~\ref{main}, and in this example we actually
have an explicit recipe of regularization: removal of the heavy rows and columns.
Note, however, that the results of \cite{FO, LV} hold for symmetric matrices as well, 
while we do not how to immediately extend Theorem~\ref{main} for symmetric matrices (the requirement of the identical distribution of entries of $A$ prevents doing simple symmetrization tricks).

Weaker versions of Theorem~\ref{main}, with an additional factor $\log n$ in the norm bound and weaker probability guarantees, can be derived from known general bounds on random matrices, such as the matrix Bernstein's inequality \cite{Tropp}. 
(One would apply the matrix Bernstein's inequality for the entries truncated at level $\sqrt{n}$, and control the larger entries as in Section~\ref{s: large entries}.)  
A different weaker bound $\|\tilde{A}\| \leq (C/\e) \sqrt{n}$, which has a suboptimal dependence on $\e$, 
can be derived in a faster way by using results of \cite{RT} directly.

\subsection*{Acknowledgement}

 Si~Tang and Antonio Auffinger were first to note a version of Proposition~\ref{twoplus}, which they kindly showed to us along with a proof based on \cite{Auff}. Our correspondence led us to add Section~\ref{s: 2plusepsilon}. We are thankful to Ramon van Handel who showed us a simple argument that we use here to prove Lemma~\ref{lem: norm L1 rows cols}. We also would like to thank the referee for the valuable suggestions, which helped to improve the presentation of our paper.

\section{The method}

Our approach to Theorem~\ref{main} utilizes and advances the methods developed recently in \cite{RT} and \cite{LV}. 
We will first control the cut norm of $A$ and then pass to the operator norm using Grothendieck-Pietsch 
factorization. Let us describe these steps in more detail.

\subsection{Three matrix norms}

The operator norm of a matrix $A$, as we already mentioned, is defined by considering $A$  
as a linear operator on the (finite dimensional) space $\ell_2$, i.e.
$$
\|A\| = \|A: \ell_2 \to \ell_2\|.
$$
Rather than bounding the operator norm of a random matrix $A$ directly, we shall compare it with two simpler norms,
\begin{gather*}
\|A\|_{\infty \to 2} = \|A: \ell_\infty \to \ell_2\| = \max_{x \ne 0} \frac{\|Ax\|_2}{\|x\|_\infty} 
\intertext{and}
\|A\|_{2 \to \infty} = \|A: \ell_2 \to \ell_\infty\| = \max_{x \ne 0} \frac{\|Ax\|_\infty}{\|x\|_2}.
\end{gather*}

The simplest of the three is the $2 \to \infty$ norm. A quick check reveals that it equals the maximum 
Euclidean norm of the rows $A_i^\tran$ of $A$:
\begin{equation}         \label{eq: 2 to infty}
\|A\|_{2 \to \infty} = \max_{i \in [n]} \|A_i\|_2.
\end{equation}
The next simplest norm is $\infty \to 2$, which can be conveniently computed as
\begin{equation}         \label{eq: infty to 2}
\|A\|_{\infty \to 2} = \max_{x \in \{-1,1\}^n} \|Ax\|_2.
\end{equation}
This norm is equivalent within a constant factor to the {\em cut norm} from the computer science literature \cite{Bol, AN},
where the maximum is taken over $\{0,1\}^n$.
The hardest of the three is the operator norm, 
\begin{equation}         \label{eq: 2 to 2}
\|A\| = \max_{x \in S^{n-1}} \|Ax\|_2.
\end{equation}
To see why the difficulty in bounding these norms rises this way, 
note that one has to control $n$ random variables in \eqref{eq: 2 to infty},
$2^n$ random variables in \eqref{eq: infty to 2}, and infinitely many random variables in \eqref{eq: 2 to 2}.

\subsection{Ideal relationships among the norms}		\label{s: ideal norms}

How large do we expect the three norms to be for random matrices? For a simple example, 
let us first consider a Gaussian random matrix $A$ with i.i.d. $N(0,1)$ entries. 
Then it is not difficult to check that
\begin{equation}         \label{eq: Gaussian}
\|A\|_{2 \to \infty} \sim \sqrt{n}, \quad 
\|A\|_{\infty \to 2} \sim n, \quad
\|A\| \sim \sqrt{n}.
\end{equation}
Indeed, note that the rows of $A$ have Euclidean norms $\sqrt{n}$ on average, so the bound on the $2 \to \infty$ norm 
follows by union bound and using Gaussian concentration. The bound on the $\infty \to 2$ norm
follows from \eqref{eq: infty to 2} by using Gaussian concentration for the normal random vector $Ax$ 
and taking the union bound over $\{-1,1\}^n$. The bound on the operator norm is a non-asymptotic version of Bai-Yin's law,
see e.g. \cite[Theorem~5.32]{V}.
 
One might wonder if \eqref{eq: Gaussian} holds not only in the Gaussian case but generally 
for random matrices $A$ with i.i.d. entries that have zero mean and unit variance.  
In particular, it would be wonderful if the three norms were always related to each other as follows:
\begin{equation}         \label{eq: three norms}
\|A\| \lesssim \frac{\|A\|_{\infty \to 2}}{\sqrt{n}} \lesssim \|A\|_{2 \to \infty} \lesssim \sqrt{n}.
\end{equation}

This, however, would be too optimistic to expect, since the bound $\|A\| \lesssim \sqrt{n}$
can not hold without higher moments assumptions as we mentioned in Section~\ref{s: local global}.
Nevertheless, we will obtain a version of \eqref{eq: three norms}
after removal a small fraction of rows of $A$. With high probability, we will be able 
to find subsets of rows $J_1 \subset J_2 \subset J_3$ with cardinalities $|J_i| \le \e n$
and such that
\begin{equation}         \label{eq: three norms relationship}
\|A_{J_3^c}\| \lesssim \frac{\|A_{J_2^c}\|_{\infty \to 2}}{\sqrt{n}} \lesssim \|A_{J_1^c}\|_{2 \to \infty} \lesssim \sqrt{n}.
\end{equation}
where the inequalities hide a factor that depends on $\e$.

\subsection{A roadmap of the proof}	

The first step in proving \eqref{eq: three norms relationship} is to find a small set $J_1$ with $|J_1| \lesssim \e n$ and
such that 
\begin{equation}         \label{eq: 2 to infty announced}
\|A_{J_1^c}\|_{2 \to \infty} \lesssim \sqrt{n}
\end{equation}
with high probability. In other words, we would like to bound all rows of $A$ simultaneously by $O(\sqrt{n})$
after removing a few columns of $A$. To show this we first focus on one row, where
we need to bound a sum of independent random variables (the squares of the row's entries).
In Theorem~\ref{thm: damping sum} we show how to bound sums of independent random variables 
almost surely by gently {\em damping} the summands. 
Damping, or reweighting down, is a softer operation than removing entries. It allows us 
to treat in Section~\ref{s: 2 to infty} all columns simultaneously without much effort, 
thus proving \eqref{eq: 2 to infty announced}. 
The argument in this step is similar to the approach proposed recently in \cite{RT}. 
We somewhat simplify the method of \cite{RT} and also improve the dependence 
between the number of removed columns and the resulting $2 \to \infty$ norm; this will ultimately
lead to the optimal dependence on $\e$ in Theorem~\ref{main}.
 
At the next step, we extend $J_1$ to a bigger set of rows $J_2$ with $|J_2| \lesssim \e n$ and so that 
\begin{equation}         \label{eq: infty to 2 desired}
\|A_{J_2^c}\|_{\infty \to 2} \lesssim n.
\end{equation}
Suppose for a moment that we are not concerned about removal of any columns. 
It is not too hard to show the general bound 
\begin{equation}         \label{eq: two norms compared announced}
\E \|A\|_{\infty \to 2} \lesssim \sqrt{n} \E \|A\|_{2 \to \infty},
\end{equation}
for a random matrix $A$ with independent, mean zero entries; 
we prove this in Lemma~\ref{lem: infty to 2 by symmetrization}.
However, this bound is not very helpful in our situation. 
We need to work with the matrix $A_{J_1^c}$ instead of $A$, which is not trivial:  
the removal of the columns in $J_1$ that we did in the first step made the entries of $A_{J_1^c}$ dependent. 
In Lemma~\ref{lem: symmetric distributions}, we first prove a variant of \eqref{eq: two norms compared announced} 
for $A_{J_1^c}$ under an additional symmetry assumption on the distribution 
of the entries of $A$. Then we manage to remove this assumption with a delicate symmetrization argument, 
which we develop in the rest of Section~\ref{s: two norms}, with the final result being Theorem~\ref{thm: infty to 2}.
The general idea of this step, as well as some of our arguments here, are inspired by \cite{RT}.
However we need to be considerably more careful than in \cite{RT} to obtain \eqref{eq: infty to 2 desired} 
with a {\em logarithmic} dependence on $\e$.

Next, we pass from $\infty \to 2$ norm to the operator norm in Section~\ref{s: bounded}. 
This is done by using Grothendieck-Pietsch factorization (Theorem~\ref{G-P}), a result that yields
the first inequality in \eqref{eq: three norms relationship} for completely arbitrary, even non-random, matrices. 
This reasoning was recently used in a similar context in \cite{LV}.

The argument we just described works under the additional assumption that the 
entries of $A$ be $O(\sqrt{n})$ almost surely. To be specific, such boundedness assumption 
is needed to make the damping argument in Step~1 work with mild, logarithmic dependence on $\e$.
The contribution of the entries that are larger than $\sqrt{n}$ are controlled in Section~\ref{s: large entries}
by showing that there can not be too many of them. 
The unit variance assumption implies that there are $O(1)$ such large entries
per column on average. This does not mean, of course, that all columns will have $O(1)$ large entries 
with high probability; in fact there could be columns with $\sim \log n / \log \log n$ large entries.
But we will check in Lemma~\ref{lem: Bernoulli} that the number of such heavy columns is small; 
removing them will lead to the desired bound $O(\sqrt{n})$ on the operator norm for the matrix with large entries. 
We develop this argument in Proposition~\ref{prop: moderate entries} and Corollary~\ref{cor: very large entries}, 
and derive the full strength of Theorem~\ref{main} in Section~\ref{s: proof main}.

Theorem~\ref{global} is proved in Section~\ref{s: global}. The paper is concluded with Section~\ref{s: questions}
where we discuss some further problems.

\subsection*{Acknowledgements}
We are thankful to Ramon van Handel who showed us a simple argument that
we use here to prove Lemma~\ref{lem: norm L1 rows cols} and to Antonio Auffinger 
for the interesting discussion of $2+\e$ finite moment case.

\section{Preliminaries}

\subsection{Notation}

Throughout the paper, positive absolute constant are denoted $C, C_1, c, c_1$, etc. Their values 
may be different from line to line. We often write $a \lesssim b$ to indicate that $a \le C b$ for 
some absolute constant $C$.

The discrete interval $\{1,2,\ldots,n\}$ is denoted by $[n]$. If $\RR$ is some subset of indices, $\RR \subset [n] \times [n]$, let us denote by $A_\RR$ the matrix obtained from $A$
by replacing the indices in $\RR$ by zero:
 $$
    A_{\RR} := (\bar{A}_{ij})_{i,j = 1}^n, \mbox{ where } \bar{A}_{ij} = A_{ij} \ind_{\{(i,j) \in \RR\}}.
$$ 
We will often consider subsets of columns of the matrix, so when $\RR = J \times [n]$ we use a simplified notation: for $J \subset [n]$
$$ A_{J} := A_{[n] \times J}.$$

Given a finite set $S$, by $|S|$ we denote its cardinality. The standard inner product in $\R^n$
shall be denoted by $\langle\cdot,\cdot\rangle$. Given $p\in[1,\infty]$, $\|\cdot\|_p$ is the standard $\ell_p^n$-norm in $\R^n$. Also, $\|\cdot\|_{\psi_2}$ denotes sub-gaussian norm of a random variable and $\|\cdot\|_{\psi_1}$ --  sub-exponential norm (see also in Section~\ref{s:conc}).

\subsection{Operator norm via $\ell_1$ norm of rows and columns}

The following simple result, known as Schur bound \cite[p. 6, \S 2]{Schur},
states that the operator norm of any matrix is dominated by the $\ell_1$ norms of rows and columns. 
For completeness, we state and prove Schur bound here; the proof is almost identical to the original one. 

\begin{lemma}				\label{lem: norm L1 rows cols}
  For any $m \times k$ matrix $A$, we have
  $$
  \|A\| \le \big( \max_i \|A_i\|_1 \cdot \max_{j} \|A^j\|_1 \big)^{1/2}
  $$
  where $A_i$ and $A^j$ denote the rows and columns of $A$.
\end{lemma}

\begin{proof}
Recall that the operator norm can be computed as a maximum of the quadratic form:
$$
\|A\| = \sup_{\|x\|_2 = \|y\|_2 = 1} |x^\tran A y|.
$$
Fix unit vectors $x$ and $y$ and express
\begin{align*}
|x^\tran A y|
&= \Big| \sum_{i,j} x_i A_{ij} y_j \Big| \\
&\le \sum_{i,j} \left( |x_i| \sqrt{|A_{ij}|} \right) \left( \sqrt{|A_{ij}|} |y_j| \right) \quad \text{(by triangle inequality)} \\
&\le \Big( \sum_{i,j} x_i^2 |A_{ij}| \Big)^{1/2} \Big( \sum_{i,j} |A_{ij}| y_j^2 \Big)^{1/2} \quad \text{(by Cauchy-Schwarz)} \\
&= \Big( \sum_i x_i^2 \, \|A_i\|_1 \Big)^{1/2} \Big( \sum_j \|A_j\|_1 \, y_j^2 \Big)^{1/2} \\
&\le \max_i \|A_i\|_1^{1/2} \cdot \max_{j} \|A^j\|_1^{1/2} \quad \text{(since $\|x\|_2 = \|y\|_2 = 1$).}
\end{align*}
Taking the maximum over all unit vectors $x$ and $y$, we complete the proof.
\end{proof}

\subsection{Concentration}\label{s:conc}

A standard way to get some desired estimate on a random variable $X$ \emph{with high probability} is to get this estimate for $\E X$ first, and then argue that $X$ \emph{concentrates} around its expectation. In this case $X$ usually stays close to $\E X$, and therefore satisfies a close estimate. 

In this paper we make use of good concentration properties of the sums of sub-gaussian (and sub-exponential) random variables, that is, such that grow not faster than standard normal (respectively, exponential) random variables. Recall that by definition a random variable $Y$ is called \emph{sub-gaussian} if its moments satisfy
$$
	\E \exp(Y^2/M_2^2) \le e,
$$
for some number $M_2 >0$. The minimal number $M_2$ is called the sub-gaussian moment of $X$, denoted as $\|Y\|_{\psi_2}$. Analogously, a random variable is called \emph{sub-exponential} if 
$$
	\E\exp(Y/M_1) \le e,
$$
for some number $M_1 >0$. The minimal number $M_1$ is called the sub-exponential moment of $Y$, denoted as $\|Y\|_{\psi_1}$. 

The class of sub-gaussian random variables contains standard normal, Bernoulli, and generally all bounded random variables. The class of sub-exponential random variables is exactly the class of squares of sub-gaussians. See \cite{V} for more information and statements of standard concentration inequalities.

Also we will need a concentration inequality for random permutations from \cite{RT}.

\begin{lemma}[Concentration for random permutations]		\label{lem: concentration permutations}
  Consider arbitrary vectors $a = (a_1,\ldots,a_n) \in \R^n$ and $x \in \{-1,1\}^n$. 
  Let $\pi : [n] \to [n]$ denote a random permutation 
  chosen uniformly from the symmetric group $S_n$. Then the random sum 
  $$
  S := \sum_{i=1}^n a_i x_{\pi(i)}
  $$
  is sub-gaussian, and 
  $$
  \|S - \E S\|_{\psi_2} \le C \|a\|_2.
  $$
  The same inequality holds for the sum $S' = \sum_{i=1}^n a_{\pi(i)} x_i$ as well, since it has the same 
  distribution as $S$.
\end{lemma}

\subsection{Discretization}

The following lemma allows us to approximate a general continuous random variable by a sum of independent, scaled Bernoulli 
random variables. This lemma was originally proved in \cite{RT}. Here we give a proof for completeness, and then discuss some 
particular cases needed for the proof of Theorem~\ref{main}.

\begin{lemma}[Discretization]				\label{lem: discretization}
  Consider a non-negative, continuous random variable $X$. 
  There exists a non-negative random variable $X'$ satisfying the following. 
  \begin{enumerate}[1.]
    \item $\E X' \le 4 \E X$.
    \item $X'$ stochastically dominates $X$, i.e. 
    $$
    \Pr{X' \ge t} \ge \Pr{X \ge t} \quad \text{for all } t \ge 0.
    $$
    
    \item $X'$ is a sum of scaled, independent Bernoulli random variables:
    \begin{equation}		\label{eq: X'}
    X' = \sum_{k=0}^\infty q_k \xi_k
    \end{equation}
    where $q_k$ are non-negative numbers 
    and $\xi_k$ are independent $\Ber(2^{-k})$ random variables. 
  \end{enumerate}
  
\end{lemma} 

\begin{proof}
Set the values $q_k$ to be the quantiles of the distribution of $X$:
$$
q_k : = \min \left\{ t \ge 0: \P\{X \ge t\} = 2^{-k-1} \right\}, \quad k = 0, 1, 2, \ldots 
$$
(These values are well defined since the cumulative distribution function of $X$ is continuous by assumption.)
By definition, $(q_k)$ is an increasing sequence.
Define $X'$ by \eqref{eq: X'}.

To check part 1, note that by definition,
\begin{equation}         \label{eq: EX'}
\E X' = \sum_{k=0}^\infty q_k \E \xi_k = \sum_{k=0}^\infty q_k 2^{-k}.
\end{equation}
To lower bound $\E X$, let us decompose $X$ according to the values it can take. This gives
$$
X \ge \sum_{k=0}^\infty X \ind_{\{X \in [q_k, q_{k+1})\}} \ge \sum_{k=0}^\infty q_k \ind_{\{X \in [q_k, q_{k+1})\}}
$$
almost surely. Taking expectation of both sides, we obtain
$$
\E X \ge \sum_{k=0}^\infty q_k \Pr{X \in [q_k, q_{k+1})}.
$$
Now, using the definition of $q_k$, we have
$$
\Pr{X \in [q_k, q_{k+1})} = \Pr{X \ge q_k} - \Pr{X \ge q_{k+1}} = 2^{-k-1} - 2^{-k-2} = 2^{-k-2}.
$$
This yields
\begin{equation}         \label{eq: EX}
\E X \ge \sum_{k=0}^\infty q_k 2^{-k-2}.
\end{equation}
Comparing \eqref{eq: EX'} with \eqref{eq: EX}, we conclude that $\E X' \le 4 \E X$, 
which proves part 1 of the lemma.

\medskip

Let us prove part 2. If $t \in [q_k, q_k+1)$ for some $k=0,1,2,\ldots$, then 
using the definitions of $X'$ and $q_k$ we obtain
\begin{align*}
\Pr{X' \ge t} 
  &\ge \Pr{X' \ge q_{k+1}}
  \ge \Pr{ \xi_{k+1} = 1 }
  = 2^{-k-1} \\
  &= \Pr{X \ge q_k}
  \ge \Pr{X \ge t},
\end{align*}
as required.

It remains to check the domination inequality when $t$ is outside the range $[q_0, q_\infty)$ where 
$q_\infty := \lim_{k \to \infty} q_k \in \R_+ \cup \{\infty\}$. 
If $t < q_0$, we have
$$
\Pr{X' \ge t} \ge \Pr{X' \ge q_0} \ge \Pr{\xi_0 = 1} = 1,
$$
and the inequality in part~2 follows. If $t \ge q_\infty$ then, using the continuity of 
the cumulative distribution of $X$, we obtain 
$$
\Pr{X \ge t} \le \Pr{X \ge q_\infty} 
= \lim_{k \to \infty} \Pr{X \ge q_k}
= \lim_{k \to \infty} 2^{-k-1} = 0,
$$
and the inequality in part~2 follows again.
The proof is complete.
\end{proof}

\begin{remark}[Bounded random variables]  \label{rem: discretization bdd}
  Suppose $X \le M$ almost surely. Then, in the second part of the conclusion of 
  Lemma~\ref{lem: discretization},
  $X$ can be represented as a {\em finite} sum 
  $$
  X' := \sum_{k=0}^\kappa q_k \xi_k
  $$
  where $q_k$ are non-negative numbers, $q_k \in[0, M]$,
  and $\xi_k$ are independent $\Ber(p_k)$ random variables.
  Here $p_k = 2^{-k} \ge 1/M$ for $k<\kappa$ and $p_\kappa = 1/M$.
\end{remark}

\begin{remark}[Coupling]		\label{rem: discretization coupling}
  Stochastic dominance of $X'$ over $X$ in   Lemma~\ref{lem: discretization} 
  implies that one can realize the random variables $X$ and $X'$ on 
  the same probability space so that 
  $$
  X' \ge X \quad \text{almost surely}.
  $$
   (See, for example, \cite[Section~4.3]{W}). 
    
  Moreover, in the same way we can construct a majorizing collection for any collection of independent random variables. 
  In particular, we can do it for all entries of the matrix $A$ at once.

\end{remark}

\section{Damping a sum of independent random variables}

Let $X_1,\ldots,X_n$ be non-negative i.i.d. random variables with $\E X_i \le 1$. 
The linearity of expectation gives the trivial bound
$$
\E \sum_{i=1}^n X_i \le n.
$$
Here we will be interested in a stronger result -- that
the sum be $O(n)$ {\em almost surely} instead of in expectation. 
To do this, we will be looking for random weights 
$$
W_1,\ldots, W_n \in [0,1]
$$
that make the ``damped'' sum satisfy
$$
\sum_{j=1}^n W_j X_j  = O(n) \quad \text{almost surely}.
$$
To make the damping as gentle as possible, we are looking for 
largest possible weights $W_i$, hopefully very close to $1$.

\subsection{Damping one random variable}

To get started, let us consider the simple case where $n=1$
and try to damp one random variable. 

\begin{lemma}[Damping a random variable]		\label{lem: damping rv}
 Let $X$ be a random variable such that 
 $$
 X \ge 0 \quad \text{and} \quad \E X \le 1. 
 $$
 Let $\e \in (0,1)$. There exists a random variable $W$ taking values in $[0,1]$ and such that 
 \begin{gather}
 XW \le \e^{-1} \quad \text{almost surely};  		\label{eq: XW bounded a.s.}\\
 1 \le \E W^{-1} \le 1 + \e.							\label{eq: gentle weight}
 \end{gather}
\end{lemma}

\begin{proof}
Fix a level $L \ge 1$ whose value we will choose later, and define
$$
W := \min(1, L/X).
$$
To check \eqref{eq: XW bounded a.s.}, we have
$$
XW = \min(X,L) \le L \quad \text{almost surely}.
$$
Next, the lower bound in \eqref{eq: gentle weight} holds trivially since $W \le 1$. 
For the upper bound, we have
$$
\E W^{-1} = \E \max(1, X/L) \le \E (1 + X/L) \le 1 + \frac{1}{L},
$$
where we used the assumption that $\E X \le 1$.
Setting $L = \e^{-1}$ completes the proof. 
\end{proof}

\subsection{Damping a sum of random variables}

Now let us address the damping problem for general number $n$ of random variables, which we described 
in the beginning of this section. Applying Lemma~\ref{lem: damping rv} for 
each random variable $X_i$, we get weights $W_i$ such that 
\begin{gather*}
\sum_{j=1}^n W_j X_j \le \e^{-1} n \quad \text{almost surely};  \\
1 \le \E \Big( \prod_{j=1}^n W_j \Big)^{-1} \le (1 + \e)^n = 1 + O(\e n)	
\end{gather*}
for small $\e$.
We will now considerably improve both these bounds, making only one mild
extra assumption that $X_i = O(n)$ almost surely. 

\begin{theorem}[Damping a sum of random variables]		\label{thm: damping sum}
 Let $X_1,\ldots,X_n$ be i.i.d. random variables such that 
 $$
 0 \le X_j \le Kn \quad \text{and} \quad \E X_j \le 1
 $$
 for some $K \ge 1$. Let $\e \in (0,1/2)$. 
 There exist random variables $W_1,\ldots,W_n$ taking values in $[0,1]$ and such that 
 \begin{gather}
 \sum_{j=1}^n W_j X_j \le C K \log (\e^{-1}) \cdot n \quad \text{almost surely};  
 	\label{eq: damped sum bdd}	\\
 1 \le \E \Big( \prod_{j=1}^n W_j \Big)^{-1} \le 1 + \e.
 	\label{eq: weights}
 \end{gather}
\end{theorem}

\begin{remark}
Improvement in the order of $n$ in \eqref{eq: weights} does not require an extra boundedness assumption, and it was done in previous work \cite{RT}. We employ the same ideas as in \cite[Lemma 3.3]{RT} and obtain better (logarithmic) dependence on $\e$ in \eqref{eq: damped sum bdd} in trade of the additional assumption mentioned.
\end{remark}
\begin{proof}

{\bf Step 1: Bernoulli distribution.}
Let us first prove the theorem in the partial case where $X_j$ are scaled 
Bernoulli random variables. Assume that $X_j$ can take values $q$ and $0$, and 
\begin{equation}         \label{eq: Bernoulli}
\Pr{X_j = q} = p \ge \frac{1}{Kn}.
\end{equation}
Let $\nu$ denote the (random) number of nonzero $X_j$'s:
$$
\nu := \left| \{ j: \; X_j \ne 0 \} \right|, \quad \text{then} \quad \E \nu = pn.
$$
Here is how we will define the weights $W_j$. If $X_j=0$ then 
clearly there is no need to damp $X_j$ so put $W_j = 1$. The same applies
if the number $\nu$ of non-zero $X_j$'s does not significantly exceed its expectation $pn$.
Otherwise we damp all terms by the same amount $W_j \sim pn/\nu$.
Formally, we fix some parameter $L = L(K,\e)$ whose value we will determine later,
and set
$$
W_j := 
\begin{cases}
1, & \text{if } \nu \le Lpn \text{ or } X_j = 0 \\
Lpn / \nu, & \text{if } \nu > Lpn \text{ and } X_j \ne 0.
\end{cases}
$$

Let us check \eqref{eq: damped sum bdd}.
In the event when $\nu \le Lpn$, we have
$$
\sum_{j=1}^n W_j X_j = \sum_{j=1}^\nu 1 \cdot q = q\nu \le qLpn
= Ln \cdot \E X_1.
$$
And in the event when $\nu > Lpn$, we have
$$
\sum_{j=1}^n W_j X_j = \sum_{j=1}^\nu \frac{Lpn}{\nu} \cdot q = Lpnq = Ln \cdot \E X_1.
$$
as before. Thus, we showed that  
\begin{equation}         \label{eq: damped sum bdd L}
\sum_{j=1}^n W_j X_j \le Ln \cdot \E X_1 \le Ln \quad \text{almost surely}.
\end{equation}

Let us now check \eqref{eq: weights}. 
Since the lower bound is trivial, we will only have to check the upper bound.
We will again split the calculation into two cases based on the size of $\nu$.
If $\nu \le Lpn$ then all $W_j = 1$, so we trivially get 
$$
E_- := \E \Big( \prod_{j=1}^n W_j \Big)^{-1} \ind_{\{\nu \le Lpn\}} \le 1.
$$
If $\nu > Lpn$, then the definition of $W_j$ gives
\begin{align*}
E_+ := \E \Big( \prod_{j=1}^n W_j \Big)^{-1} \ind_{\{\nu > Lpn\}}
&= \E \Big( \frac{\nu}{Lpn} \Big)^\nu \ind_{\{\nu > Lpn\}} \\
&= \sum_{k=\lceil Lpn \rceil+1}^n \Big( \frac{k}{Lpn} \Big)^k \Pr{\nu = k}.
\end{align*}
Since $\nu \sim \Binom(n,p)$, we have
$$
\Pr{\nu = k} = \binom{n}{k} p^k \le \Big( \frac{enp}{k} \Big)^k,
$$
using a standard consequence of Stirling's approximation.
Thus 
$$
E_+ \le \sum_{k=\lceil Lpn \rceil+1}^n \Big( \frac{e}{L} \Big)^k
\le \Big( \frac{e}{L} \Big)^{Lpn},
$$
provided that $L \ge 10$. 
Thus we showed that 
\begin{equation}         \label{eq: weights L}
\E \Big( \prod_{j=1}^n W_j \Big)^{-1} 
\le E_- + E_+ 
\le 1 +\Big( \frac{e}{L} \Big)^{Lpn}
\le 1 + \Big( \frac{e}{L} \Big)^{L/K}
\end{equation}
where in the last step we used the assumption that $p \ge 1/Kn$ 
that we made in \eqref{eq: Bernoulli}.

Now that we have the bounds \eqref{eq: damped sum bdd L} and \eqref{eq: weights L}, 
it is enough to choose 
$$
L := C K \log \Big(\frac{1}{\e}\Big)
$$
which implies that $E \le 1+\e$. The proof for the Bernoulli distribution is complete.

\bigskip

{\bf Step 2. General distribution.}
Let us now now prove the theorem in full generality. 
First we discretize the distribution of $X_j$ using Lemma~\ref{lem: discretization}.
This result requires $X_j$ be continuous, which can be arranged by a standard approximation argument. 
For example, we can add a small Gaussian independent component to $X_j$ and then let the variance of this component
go to zero.  
Taking into account Remarks~\ref{rem: discretization bdd} and \ref{rem: discretization coupling}, we obtain independent, non-negative random variables $X'_j$ that satisfy 
$\E X'_j \le 4$ and such that 
$$
X_j \le X'_j = \sum_{k=1}^\kappa X_{jk}.
$$
Here $X_{jk}$ are independent random variables; each $X_{jk}$ can take 
values $q_k$ and $0$, and  
$$
\Pr{X_{jk} = q_k} = p_k
$$
with  
\begin{equation}         \label{eq: pk}
p_k = 2^{-k} \ge \frac{1}{Kn} \text{ for } k < \kappa, 
\quad p_\kappa = \frac{1}{Kn}.
\end{equation}

The argument will be similar to step~1 of the proof. For each level $k$ we let $\nu_k$ denote
number of non-zero $X_{jk}$'s:
$$
\nu_k := \left| \{ j: \; X_{jk} \ne 0 \} \right|, \quad \text{then} \quad \E \nu = p_k n.
$$
Again, for each level $k$ define the weights $W_{jk}$ like in step~1:
$$
W_{jk} := 
\begin{cases}
1, & \text{if } \nu_k \le Lp_k n \text{ or } X_{jk} = 0 \\
Lp_k n / \nu_k, & \text{if } \nu_k > Lp_k n \text{ and } X_{jk} \ne 0.
\end{cases}
$$
Then we set 
$$
W_j := \prod_{k=1}^\kappa W_{jk}, \quad j=1,\ldots,n.
$$

Let us check \eqref{eq: damped sum bdd}.
We have 
\begin{equation}         \label{eq: sum Wj Xj stratified}
\sum_{j=1}^n W_j X_j 
\le \sum_{j=1}^n W_j X'_j 
= \sum_{j=1}^n \sum_{k=1}^\kappa W_j X_{jk}  
\le \sum_{k=1}^\kappa \sum_{j=1}^n W_{jk} X_{jk},
\end{equation}
since $W_j \le W_{jk}$ by construction. Now, for each level $k$, we can 
use step~1 of the proof, where we showed in \eqref{eq: damped sum bdd L} that 
$$
\sum_{j=1}^n W_{jk} X_{jk} \le Ln \cdot \E X_{1k}.
$$
Substituting into \eqref{eq: sum Wj Xj stratified}, we obtain 
\begin{equation}         \label{eq: damped sum bdd L general}
\sum_{j=1}^n W_j X_j \le Ln \cdot \sum_{k=1}^\kappa \E X_{1k}
= Ln \cdot \E X'_1 \le 5 Ln 
\end{equation}
by construction.

Let us now check \eqref{eq: weights}. 
The lower bound is trivial, and we will only have to check the upper bound.
For each level $k$, we can use step~1 of the proof, where we showed in \eqref{eq: weights L}
that 
$$
\E \Big( \prod_{j=1}^n W_{jk} \Big)^{-1} 
\le 1 +\Big( \frac{e}{L} \Big)^{Lp_k n}
\le 1 + e^{-Lp_k n},
$$
which is true as long as $L \ge 10$.
Then, by construction we have
\begin{align*}
\E \Big( \prod_{j=1}^n W_j \Big)^{-1}
&= \E \prod_{k=1}^\kappa \Big( \prod_{j=1}^n W_{jk} \Big)^{-1} \\
&= \prod_{k=1}^\kappa \E \Big( \prod_{j=1}^n W_{jk} \Big)^{-1} 
	\quad \text{(by independence)} \\
&\le \prod_{k=1}^\kappa \left( 1 + e^{-Lp_k n} \right)
\le \exp \Big( \sum_{k=1}^\kappa e^{-Lp_k n} \Big)
\end{align*}
where in the last step we used the inequality $1+x \le e^x$.
Recall from \eqref{eq: pk} that the exponents $p_k$ form a decreasing geometric 
progression with values $2^{-k}$ until the last (smallest) term of order $1/Kn$.
So this last term dominates the sum $\sum_{k=1}^\kappa e^{-Lp_k n}$, and we obtain 
\begin{equation}         \label{eq: weights L general}
\E \Big( \prod_{j=1}^n W_j \Big)^{-1} \le \exp( 2e^{-L/2K} ).
\end{equation}

Now that we have the bounds \eqref{eq: damped sum bdd L general} and \eqref{eq: weights L general}, 
it is enough to choose 
$$A_{ij}
L := C_{\ref{thm:  damping sum}} K \log \Big(\frac{1}{\e}\Big)
$$
with $C_{\ref{thm:  damping sum}} \ge 6K$ and the right hand side of \eqref{eq: weights L general} will be bounded by 
$$
\exp(2 \e^3) \le \exp(\e/2) \le 1+\e,
$$
as claimed. The proof of the theorem is complete.
\end{proof}

\section{The $2 \to \infty$ norm of random matrices}		\label{s: 2 to infty}

In this section we prove Theorem~\ref{main} under the additional assumption 
that all entries $A_{ij}$ of $A$ are not too large. Specifically, let us assume that 
\begin{equation}	\label{trunc_probab_model}
  |A_{ij}| \le \frac{\sqrt{n}}{2} \quad \text{almost surely}.
\end{equation}

\begin{lemma}[Bounding $2 \to \infty$ norm by removing a few columns]		\label{rows_cut}
  Consider an $n \times n$ random matrix $A$ with i.i.d. entries $A_{ij}$ which have mean zero and at most unit variance 
  and satisfy \eqref{trunc_probab_model}. Let $\e \in (0,1/2]$. 
  Then with probability at least $1 - \exp(-\e n)$, there exists a subset $J \in [n]$ with cardinality $|J| \le \e n$ 
 such that 
  $$
  \|A_{J^c}\|_{2 \to \infty} \le C \sqrt{\ln \e^{-1}} \cdot \sqrt{n}.
  $$
 \end{lemma}

\begin{proof}
We apply Theorem~\ref{thm:  damping sum} for the squares of the elements in each row of $A$, i.e. for 
the random variables $(a_{i1}^2, \ldots, a_{in}^2)$. This gives us random weights $W_{ij} \in [0,1]$ which 
satisfy for each $i \in [n]$ that 
$$
\sum_{j=1}^n W_{ij} A_{ij}^2 \le C \log(\e^{-1}) n \quad \text{a.s.}; \qquad
\E \Big( \prod_{j=1}^n W_{ij} \Big)^{-1} \le \exp(\e).
$$
To make the same system of weights work for all rows, we define
$$
V_j := \prod_{i=1}^n W_{ij}\in [0,1], \quad j \in [n].
$$
Then obviously $V_j \le W_{ij}$ for every $i$, and so
\begin{equation}         \label{eq: Vj}
\sum_{j=1}^n V_j A_{ij}^2 \le C \log(\e^{-1}) n \;\; \forall i \quad \text{a.s.}; \qquad
\E \Big( \prod_{j=1}^n V_j \Big)^{-1} \le \exp(\e n).
\end{equation}

We will remove from $A$ the columns whose weights $V_j$ are too small, namely those in
$$
J := \{ j \in [n]: \; V_j < e^{-2} \}.
$$

Let us first check that 
\begin{equation}         \label{eq: J0 cardinality}
|J| \le \e n \quad \text{with probability at least } 1 - \exp(-\e n),
\end{equation}
as we claimed in the lemma.
Indeed, if $|J| > \e n$ then using that all $V_j \in[0,1]$ we have
$$
Z := \prod_{j=1}^n V_j \le \prod_{j \in J} V_j < e^{-2\e n}.
$$
But the probability of this event can be bounded by Markov's inequality: 
$$
\Pr{ Z < e^{-2 \e n} } 
= \Pr{ Z^{-1} > e^{2 \e n} }
\le e^{-2 \e n} \E Z^{-1}
\le e^{-\e n},
$$
where in the last bound we used \eqref{eq: Vj}. This proves \eqref{eq: J0 cardinality}.

It remains to check that all rows $B_i$ of the matrix $B = A_{[n] \times J_0^c}$ are bounded
as claimed. We have 
\begin{align*}
\|B_i\|_2^2 = \sum_{j \in J^c} A_{ij}^2
  &\le e^2 \sum_{j \in J^c} V_j A_{ij}^2 \quad \text{(by definition of $J$)} \\
  &\le e^2 \sum_{j=1}^n V_j A_{ij}^2 \quad \text{(since all $V_j \le 1$)} \\  
  &\le e^{2} C \ln(\e^{-1}) n \quad \text{(by \eqref{eq: Vj})}.
\end{align*}
Taking the square root of both sides completes the proof.
\end{proof}

\section{From $2 \to \infty$ norm to $\infty \to 2$ norm}			\label{s: two norms}

In this section we will control the $\infty \to 2$ norm of a random matrix. 
Our first task is to bound the $\infty \to 2$ norm by the simpler $2 \to \infty$ norm.  
There are two ways to do this, both of them going back to \cite{RT}. The resulting 
comparison inequalities are interesting in their own right; we state them 
in Lemmas~\ref{lem: infty to 2 by symmetrization} and \ref{lem: infty to 2 by permutations}.
The ultimate result of this section is Theorem~\ref{thm: infty to 2}, which gives an optimal bound $O(n)$
on the $\infty \to 2$ norm of a random matrix after removing a small fraction of columns.

\subsection{Using random signs}		

The first method is based on flipping the signs of the entries independently at random. 
Here is the main result of this section.

\begin{lemma}[From $2 \to \infty$ to $\infty \to 2$]		\label{lem: infty to 2 by symmetrization}
  Let $A$ be an $n \times n$ random matrix whose entries are independent, mean zero random variables. 
  Then 
  $$
  \E \|A\|_{\infty \to 2} \le C \sqrt{n} \cdot \E \|A\|_{2 \to \infty}.
  $$ 
\end{lemma}

\begin{proof}
Let $\e_{ij}$ be independent Rademacher random variables (which are also independent of $A$)
and consider the random matrix
$$
\tilde{A} := (\e_{ij} A_{ij}).
$$
A basic symmetrization inequality (see \cite[Lemma~6.3]{LT}) yields
$$
\E \|A\|_{\infty \to 2} \le 2 \E \|\tilde{A}\|_{\infty \to 2}.
$$
Condition on $A$; the randomness now rests in the random signs $(\e_{ij})$ only.
It suffices to show that the conditional expectation satisfies
\begin{equation}         \label{eq: A tilde desired}
\E \|\tilde{A}\|_{\infty \to 2} \lesssim \sqrt{n} \cdot \|A\|_{2 \to \infty}.
\end{equation}
Recalling \eqref{eq: infty to 2}, we have
\begin{equation}         \label{eq: Atilde infty to 2}
\|\tilde{A}\|_{\infty \to 2} = \max_{x \in \{-1,1\}^n} \|\tilde{A}x\|_2. 
\end{equation}
According to the matrix-vector multiplication, we can express $\|\tilde{A}x\|^2$ 
as a sum of independent random variables
$$
\|\tilde{A}x\|_2^2 = \sum_{i=1}^n \xi_i^2 \quad \text{where} \quad 
\xi_i := \ip{\tilde{A}_i}{x} = \sum_{j=1}^n \e_{ij} A_{ij} x_j.
$$
Fix $x \in \{-1,1\}^n$. Using independence and \eqref{eq: 2 to infty}, we get 
$$
\E  \xi_i^2 = \sum_{j=1}^n (A_{ij} x_{ij})^2 = \sum_{j=1}^n A_{ij}^2 \le \|A\|_{2 \to \infty}^2,
$$
so 
\begin{equation}         \label{eq: Atilde expectation}
\E \sum_{i=1}^n \xi_i^2 \le n \|A\|_{2 \to \infty}^2.
\end{equation}
Moreover, the standard concentration results (\cite[Lemma~5.9]{V}) show that
each $\xi_i$ is a sub-gaussian random variable, and we have
$$
\|\xi_i\|_{\psi_2}^2 = \Big\| \sum_{j=1}^n \e_{ij} A_{ij} x_j \Big\|_{\psi_2}^2
\lesssim \sum_{j=1}^n (A_{ij} x_{ij})^2 \le \|A\|_{2 \to \infty}^2.
$$
Thus $\xi_i^2$ is a sub-exponential random variable (see \cite[Lemma~5.9]{V}) and 
\begin{equation}         \label{eq: Atilde sub-exponential}
\|\xi_i^2\|_{\psi_1} \lesssim \|\xi_i\|_{\psi_2}^2 \lesssim \|A\|_{2 \to \infty}^2.
\end{equation}

Applying Bernstein's concentration inequality \cite[Corollary~5.17]{V} together with \eqref{eq: Atilde expectation}
and \eqref{eq: Atilde sub-exponential}, we obtain
$$
\Pr{ \sum_{i=1}^n \xi_i^2 \ge n \|A\|_{2 \to \infty}^2 + t n \|A\|_{2 \to \infty}^2 }
\le \exp(-ctn)
$$
for all $t \ge 1$. Thus we obtained a bound on $\|\tilde{A}x\|_2^2 = \sum_{i=1}^n \xi_i^2$. 
It remains to recall \eqref{eq: Atilde infty to 2} and take a union bound over $x \in \{-1,1\}^n$.
It follows that the inequality
$$
\|\tilde{A}\|_{\infty \to 2}^2 \le (1+t) n \|A\|_{2 \to \infty}^2
$$
holds with probability at least
$$
1 - 2^n \exp(-ctn) \ge 1- \exp \left[ (1-ct)n \right],
$$
where $t \ge 1$ is arbitrary.
Integration of these tails implies \eqref{eq: A tilde desired}.
\end{proof}

We will need a minor variation of Lemma~\ref{lem: infty to 2 by symmetrization} 
that can be applied even when some of the columns of $A$ are removed. 

\begin{lemma}[From $2 \to \infty$ to $\infty \to 2$ for symmetric distributions]		\label{lem: symmetric distributions}
  Let $A$ be an $n \times n$ random matrix whose entries are independent, symmetric random variables.
  Let $J \subset [n]$ be a random subset, which is independent of the signs of the entries of $A$.
  Then 
  $$
  \|A_{J}\|_{\infty \to 2} \le C \sqrt{n} \|A_{J}\|_{2 \to \infty}
  $$
  with probability at least $1-e^{-n}$.
\end{lemma}

\begin{proof}
It is quite straightforward to check this result by modifying the proof of Lemma~\ref{lem: infty to 2 by symmetrization}.
By the symmetry assumption, the matrix $\tilde{A} := (\e_{ij} A_{ij})$ has the same distribution as $A$. 
Conditioning on $A$ and $J$ leaves all randomness with the signs $(\e_{ij})$, as before. 
Then we repeat the reset of the proof of  Lemma~\ref{lem: infty to 2 by symmetrization} for the submatrix $A_J$, 
In the end, we choose $t$ to be a large absolute constant to complete the proof. 
\end{proof}
So, the only part of Lemma \ref{lem: infty to 2 by symmetrization} that does not work for a matrix with removed columns is the symmetrization part. In the following two sections we will develop the tools to overcome the extra symmetry assumption we have to add in Lemma \ref{lem: symmetric distributions}.

\subsection{Using random permutations}			\label{s: random permutations}

We just showed how to convert an $\infty \to 2$ bound to a $2 \to \infty$ bound for random matrices by using 
random signs. Alternatively, one can use random permutations for the same purpose, and obtain the following bound.

\begin{lemma}[From $2 \to \infty$ to $\infty \to 2$]		\label{lem: infty to 2 by permutations}
  Let $A$ be an $n \times n$ random matrix with i.i.d. entries. 
  Then 
  $$
  \E \|A\|_{\infty \to 2} \le C \sqrt{n} \cdot \E \|A\|_{2 \to \infty} + C \E \|A \one\|_2,
  $$ 
  where $\one = (1,1,\ldots,1)$ denotes the vector whose all coordinates equal $1$.
\end{lemma}

Before we turn to the proof, note that the only difference between Lemmas~\ref{lem: infty to 2 by symmetrization} and 
\ref{lem: infty to 2 by permutations} is the term $\E \|A \one\|_2$. It makes its appearance since there is no mean zero 
assumption on the entries. This term is usually quite innocent. Note also that \eqref{eq: infty to 2} trivially implies that
$$
\E \|A\|_{\infty \to 2} \ge \E \|A \one\|_2,
$$
so we have to control this term anyway.

\begin{proof}
Let us apply a random independent permutation $\pi_i$ to the elements of each row of $A$.
The resulting matrix $\tilde{A}$ has the same distribution of $A$ due to the i.i.d. assumption.
Condition on $A$; the randomness now rests in the random permutations $\pi_i$ only. 
It suffices to show that the conditional expectation satisfies 
\begin{equation}         \label{eq: A tilde desired permutations}
\E \|\tilde{A}\|_{\infty \to 2} \le C \sqrt{n} \cdot \|A\|_{2 \to \infty} + C \|A \one\|_2,
\end{equation}
Similarly to the proof of Lemma~\ref{lem: infty to 2 by symmetrization}, we express $\|\tilde{A}x\|^2$ 
as a sum of independent random variables
\begin{equation}         \label{eq: Ax expressed permutations}
\|\tilde{A}x\|_2^2 = \sum_{i=1}^n \xi_i^2 \quad \text{where} \quad 
\xi_i := \ip{\tilde{A}_i}{x} = \sum_{j=1}^n A_{i, \pi_i(j)} x_j.
\end{equation}
The concentration inequality for random permutations (Lemma~\ref{lem: concentration permutations})
states that each $\xi_i$ is a sub-gaussian random variable, and we have
$$
\|\xi_i - \E \xi_i\|_{\psi_2} \lesssim \|\tilde{A}_i\|_2 \le \|A\|_{2 \to \infty}.
$$
Just like in the proof of Lemma~\ref{lem: infty to 2 by symmetrization}, this implies that 
$$
\|(\xi_i - \E \xi_i)^2\|_{\psi_1} \lesssim \|A\|_{2 \to \infty}^2.
$$
Since the expectation is bounded by the $\psi_1$ norm (see e.g. \cite[Definition~5.13]{V}), we conclude that
$$
\E (\xi_i - \E \xi_i)^2 \lesssim \|(\xi_i - \E \xi_i)^2\|_{\psi_1} \lesssim \|A\|_{2 \to \infty}^2
$$
and thus
$$
\E \sum_{i=1}^n (\xi_i - \E \xi_i)^2 \lesssim n \|A\|_{2 \to \infty}^2.
$$
Applying Bernstein's inequality like in Lemma~\ref{lem: infty to 2 by symmetrization}, we find that
$$
\Pr{ \sum_{i=1}^n (\xi_i - \E \xi_i)^2 \ge n \|A\|_{2 \to \infty}^2 + t n \|A\|_{2 \to \infty}^2 }
\le \exp \left[ -c \min(t^2,t) n \right]
$$
for all $t \ge 0$. 
Thus, for any $t \ge 1$ we have with probability at least $1-\exp(-tn)$ that
\begin{equation}         \label{eq: centered sum bounded}
\sum_{i=1}^n (\xi_i - \E \xi_i)^2 \le (1+t) n \|A\|_{2 \to \infty}^2.
\end{equation}

From \eqref{eq: Ax expressed permutations} we see that we are almost done; 
we just need to remove $\E \xi_i$ from our bound. To this end, note that
\begin{equation}         \label{eq: Ax bounded two terms}
\|\tilde{A}x\|_2^2 
= \sum_{i=1}^n \xi_i^2 
\le 2 \sum_{i=1}^n (\xi_i - \E \xi_i)^2 + 2 \sum_{i=1}^n (\E \xi_i)^2.
\end{equation}
We have already bounded the first sum. As for the second one, the definition of $\xi$ in \eqref{eq: Ax expressed permutations}
yields
$$
\E \xi_i = \frac{2m-n}{n} \sum_{j=1}^n A_{ij} = \frac{2m-n}{n} \ip{A_i}{\one}
$$
where $m$ denotes the number of ones in $x_j$ and $A_i^\tran$ is the $i$-th row of $A$. Thus 
$$
\sum_{i=1}^n (\E \xi_i)^2 = \Big( \frac{2m-n}{n} \Big)^2 \sum_{i=1}^n \ip{A_i}{\one}^2 
\le \|A \one\|_2^2.
$$
We substitute this and \eqref{eq: centered sum bounded} into \eqref{eq: Ax bounded two terms} and obtain
that for any $t \ge 1$, 
$$
\|\tilde{A}x\|_2^2 \le 2(1+t) n \|A\|_{2 \to \infty}^2 + 2\|A \one\|_2^2
$$
with probability at least $1-\exp(-tn)$.

It remains to recall \eqref{eq: Atilde infty to 2} and take a union bound over $x \in \{-1,1\}^n$.
It follows that the inequality
\begin{equation}         \label{eq: Atilde bounded permutations whp}
\|\tilde{A}\|_{\infty \to 2}^2 \le  2(1+t) n \|A\|_{2 \to \infty}^2 + 2\|A \one\|_2^2
\end{equation}
holds with probability at least 
$$
1 - 2^n \exp(-ctn) \ge 1-\exp \left[ (1-ct)n \right]
$$
where $t \ge 1$ is arbitrary. Integration of these tails implies \eqref{eq: A tilde desired permutations}.
\end{proof}

It is worthwhile to mention a high-probability version of Lemma~\ref{lem: infty to 2 by permutations}. 

\begin{lemma}[From $2 \to \infty$ to $\infty \to 2$ with high probability]		\label{lem: infty to 2 by permutations high prob}
  Let $A$ be an $n \times n$ random matrix with i.i.d. entries. 
  Then with probability at least $1-e^{-n}$ we have
  $$
  \|A\|_{\infty \to 2} \le C \sqrt{n} \cdot \E \|A\|_{2 \to \infty} + C \E \|A \one\|_2,
  $$ 
  where $\one = (1,1,\ldots,1)$ denotes the vector whose all coordinates equal $1$.
\end{lemma}

\begin{proof}
At the end of the proof of Lemma~\ref{lem: infty to 2 by permutations}, we obtained 
inequality \eqref{eq: Atilde bounded permutations whp} which states (for large constant $t$) that
$$
\|\tilde{A}\|_{\infty \to 2} \le C \sqrt{n} \cdot \|A\|_{2 \to \infty} + C \|A \one\|_2
$$
with probability at least $1-e^{-n}$. 
Note that 
$$
\|A\|_{2 \to \infty} = \|\tilde{A}\|_{2 \to \infty} \quad \text{and} \quad
\|A \one\|_2 = \|\tilde{A} \one\|_2
$$
deterministically. Indeed, it is easy to check that permutations of the elements of the rows of $A$ do not
affect these two quantities. It follows that
$$
\|\tilde{A}\|_{\infty \to 2} \le C \sqrt{n} \cdot \|\tilde{A}\|_{2 \to \infty} + C \|\tilde{A} \one\|_2
$$
with probability at least $1-e^{-n}$. 
It remains to note that $\tilde{A}$ has the same distribution as $A$.
\end{proof}

\subsection{Bounding $2 \to \infty$ and $\infty \to 2$ norms with tiny probability}

Recall from Section~\ref{s: ideal norms} that ideally, we would want
$$
\|A\|_{2 \to \infty} \lesssim \sqrt{n} \quad \text{and} \quad
\|A\|_{\infty \to 2} \lesssim n
$$
with high probability. But this is too good to be true in our situation, where we assume only 
two moments for the entries of $A$. Nevertheless, we will now show that these bounds
still hold, albeit with exponentially small probability.

\begin{lemma}[$2 \to \infty$ and $\infty \to 2$ norms with tiny probability]		\label{lem: tiny probability}
  Let $A$ be an $n \times n$ random matrix whose entries are i.i.d. random variables
  with mean zero and at most unit variance. Let $\d \in (0,1/2)$. Then 
  \begin{equation}         \label{eq: tiny probability}
  \|A\|_{2 \to \infty} \le 2 \d^{-1} \sqrt{n} 
  \quad \text{and} \quad
  \|A\|_{\infty \to 2} \le C \d^{-1} n
  \end{equation}
  with probability at least $\frac{1}{2} \exp(-\d^2 n)$.
\end{lemma}

\begin{proof}
We will first bound below the probability of the event 
$$
\EE := \left\{ \|A\|_{2 \to \infty} \le 2 \d^{-1} \sqrt{n}
\text{ and }
\|\tilde{A} \one\|_2 \le 2 \d^{-1} n \right\}
$$
and then use Lemma~\ref{lem: infty to 2 by permutations high prob} to 
control $\|A\|_{\infty \to 2}$. 

Recall from \eqref{eq: 2 to infty} that 
$$
\|A\|_{2 \to \infty} = \max_{i \in [n]} \|A_i\|_2
\quad \text{and} \quad
\|\tilde{A} \one\|_2^2 = \sum_{i=1}^n \ip{A_i}{\one}^2
$$
where $A_i^\tran$ denote the rows of $A$. Thus 
$\EE \subset \bigcap_{i=1}^n \EE_i$
where 
$$
\EE_i := \left\{ \|A_i\|_2 \le 2 \d^{-1} \sqrt{n}
\text{ and }
|\ip{A_i}{\one}| \le 2 \d^{-1} \sqrt{n} \right\}
$$
are independent events.
This reduces the problem to bounding the probability of each event $\EE_i$ below.

The assumptions on the entries of $A$ imply that 
$$
\E \|A_i\|_2^2 \le n
\quad \text{and} \quad
\E \ip{A_i}{\one}^2 \le n.
$$
Using Chebyshev's inequality, we see that
$$
\Pr{ \|A_i\|_2 > 2 \d^{-1} \sqrt{n} } \le \frac{\d^2}{4}
\quad \text{and} \quad
\Pr{ |\ip{A_i}{\one}| > 2 \d^{-1} \sqrt{n} } \le \frac{\d^2}{4}.
$$
Then a union bound yields
$$
\P(\EE_i) \ge 1 - \frac{\d^2}{2}.
$$
By independence of the events $\EE_i$, this implies
$$
\P(\EE) \ge \Big( 1 - \frac{\d^2}{2} \Big)^n
\ge \exp(-\d^2 n).
$$

Next we apply Lemma~\ref{lem: infty to 2 by permutations high prob}, which states that
the event 
$$
\FF := \left\{ \|A\|_{\infty \to 2} \le C \sqrt{n} \cdot \E \|A\|_{2 \to \infty} + C \E \|A \one\|_2 \right\}
$$
is likely: 
$$
\P(\FF) \ge 1 - \exp(-n).
$$
It follows that
$$
\P(\EE \cap \FF) \ge \exp(-\d^2 n) - \exp(-n) \ge \frac{1}{2} \exp(-\d^2 n).
$$
It remains to note that by definition of $\EE$ and $\FF$, 
the event $\EE \cap \FF$ implies the inequalities in \eqref{eq: tiny probability}.
\end{proof}

\subsection{Bounding $\infty \to 2$ norm with high probability}

In the previous section, we were able to prove the optimal bounds
$$
\|A\|_{2 \to \infty} \lesssim \sqrt{n} \quad \text{and} \quad
\|A\|_{\infty \to 2} \lesssim n
$$
for a random matrix $A$, but they only hold with exponentially small probability. 
We claim that the probability of success can be increased to almost $1$ if we are allowed to 
remove a few columns of $A$. We already proved this fact for the $2 \to \infty$ norm in Lemma~\ref{rows_cut}. 
It is time to handle the $\infty \to 2$ norm. 

\begin{theorem}[Bounding $\infty \to 2$ norm by removing a few columns]				\label{thm: infty to 2}
  Consider an $n \times n$ random matrix $A$ with i.i.d. entries $A_{ij}$ which have mean zero and at most unit variance 
  and satisfy \eqref{trunc_probab_model}. Let $\e \in (0,1/2]$. 
  Then with probability at least $1 - 2\exp(-\e n)$, there exists a subset $J \in [n]$ with cardinality $|J| \le \e n$ 
  such that 
  $$
  \|A_{J^c}\|_{\infty \to 2} \le C \sqrt{\ln \e^{-1}} \cdot n.
  $$
\end{theorem}	

\begin{proof}
{\bf Step 1: Defining the two key events.}
We will be interested in the two key events that suitably control the $2 \to \infty$ and $\infty \to 2$ norms 
of a random matrix. Thus, for a random matrix $B$ and numbers $r,K \ge 0$, we define
\begin{align*}
\EE_{2 \to \infty}(B,r,K) &:= \left\{ \exists J, \, |J| \le r \e n: \; \|B_{J^c}\|_{2 \to \infty} \le K \sqrt{\ln \e^{-1}} \cdot \sqrt{n} \right\}, \\
\EE_{\infty \to 2}(B,r,K) &:= \left\{ \exists J, \, |J| \le r \e n: \; \|B_{J^c}\|_{\infty \to 2} \le K \sqrt{\ln \e^{-1}} \cdot n \right\}.
\end{align*}
In terms of these events, we want to show that 
$$
\P \left( \EE_{\infty \to 2}(A,1,C)^c \right) \le 2\exp(-\e n),
$$
while Lemma~\ref{rows_cut} can be stated as 
$$
\P \left( \EE_{2 \to \infty}(A,1,C') \right) \ge 1 - \exp(-\e n).
$$
for some absolute constant $C'$.
Since the latter event is so likely, intersecting with it would not cause much harm.
Indeed, we will show that the bad event 
$$
\BB := \EE_{2 \to \infty}(A,1,C') \cap \EE_{\infty \to 2}(A,1,C)^c
$$
satisfies 
\begin{equation}         \label{eq: prob BB}
\P(\BB) \le \exp(-n/2).
\end{equation}
This would finish the proof, since we would then have 
$$
\P \left( \EE_{\infty \to 2}(A,1,C)^c \right) \le \exp(-n/2) + \exp(-\e n) \le 2\exp(-\e n)
$$
as required.

\medskip

{\bf Step 2: Symmetrization.}
As an intermediate step, let us bound the probability of a symmetrized version of $\BB$, namely the event
$$
\tilde{\BB} := \EE_{2 \to \infty}(\tilde{A},1,2C') \cap \EE_{\infty \to 2}(\tilde{A},1,C/2)^c
$$
where 
$$
\tilde{A} := A - A'
$$
and $A'$ is an independent copy of the random matrix $A$.
We claim that 
\begin{equation}         \label{eq: prob Btilde}
\P(\tilde{\BB}) \le \exp(-n).
\end{equation}

To prove this claim, choose a subset $J$, $|J| \le \e n$, that minimizes $\|\tilde{A}_{J^c}\|_{2 \to \infty}$.
Recall from \eqref{eq: 2 to infty} that the $2 \to \infty$ norm of a matrix is determined by the Euclidean norms 
of the columns and thus does not depend on the signs of the matrix elements. Thus $J$ is independent 
of the signs of the elements of $\tilde{A}$. This makes it possible to use Lemma~\ref{lem: symmetric distributions}
for the matrix $\tilde{A}$ and the random set $J^c$. It gives 
\begin{equation}         \label{eq: part of B}
\|\tilde{A}_{J^c}\|_{\infty \to 2} \lesssim \sqrt{n} \|\tilde{A}_{J^c}\|_{2 \to \infty}
\end{equation}
with probability at least $1-\exp(-n)$.

Then, turning to $\tilde{\BB}$, we can bound its probability as follows: 
$$
\P(\tilde{\BB}) \le \P(\tilde{\BB} \text{ and \eqref{eq: part of B}}) + \exp(-n).
$$
To prove the claim, it remains to check that $\tilde{\BB}$ and \eqref{eq: part of B} can not hold together.
Assume they do; then 
$$
\|\tilde{A}_{J^c}\|_{\infty \to 2} \lesssim \sqrt{n} \cdot 2C' \sqrt{\ln \e^{-1}} \sqrt{n} 
\lesssim \sqrt{\ln \e^{-1}} \cdot n, 
$$
which contradicts the event $\EE_{\infty \to 2}(\tilde{A},1,C/2)^c$ in the definition of $\tilde{\BB}$
for a suitably chosen constant $C$. This completes the proof of the claim \eqref{eq: prob Btilde}.

\medskip

{\bf Step 3. Using the small-probability bounds.}
The last piece of information we will use is the conclusion of Lemma~\ref{lem: tiny probability}
for $\d := 1/(2\ln \e^{-1})$. It states that the good event 
$$
\GG := \EE_{2 \to \infty}(A',0,C') \cap \EE_{\infty \to 2}(A',0,C/2)
$$
is likely to happen: 
\begin{equation}         \label{eq: prob GG}
\P(\GG) \ge \frac{1}{2} \exp \Big( - \frac{n}{4 \ln \e^{-1}} \Big).
\end{equation}
Note in passing that there is no guarantee that this statement would hold for the same constants $C$ and $C'$
as we chose in the definition of $\BB$ above. However, we can make this happen by adjusting these constants upwards
as necessary. The reader can easily check both \eqref{eq: prob Btilde} and \eqref{eq: prob GG} would still hold after 
such an adjustment.

We claim that 
\begin{equation}         \label{eq: BB GG inclusion}
\BB \cap \GG \subset \tilde{\BB}.
\end{equation}
To see this, recall that each of $\BB$, $\GG$ and $\tilde{\BB}$ is defined an an intersection of two 
events, one controlling $2 \to \infty$ norm and the other, $\infty \to 2$ norm. Thus it suffices to check the
inclusion for each of these two parts separately. Namely, the claim \eqref{eq: BB GG inclusion} 
would follow at once if we show that
\begin{align*}
\EE_{2 \to \infty}(A,1,C') &\cap \EE_{2 \to \infty}(A',0,C') \subset \EE_{2 \to \infty}(\tilde{A},1,2C') \quad \text{and} \\
\EE_{\infty \to 2}(A,1,C)^c &\cap \EE_{\infty \to 2}(A',0,C/2) \subset \EE_{\infty \to 2}(\tilde{A},1,C/2)^c.
\end{align*}
Both these inclusions are straightforward to check from the definitions of the events $\EE_{2 \to \infty}$ and $\EE_{\infty \to 2}$, 
remembering that $\tilde{A} = A-A'$ and using triangle triangle inequality. This verifies the claim \eqref{eq: BB GG inclusion}.

The event $\BB$ is determined by $A$, and $\GG$ is determined by $A'$ only. Thus $\BB$ and $\GG$ are independent, 
and \eqref{eq: BB GG inclusion} gives
$$
\P(\BB) \, \P(\GG) = \P(\BB \cap \GG) \le \PP(\tilde{\BB}).
$$
Thus, using \eqref{eq: prob Btilde} and \eqref{eq: prob GG}, we conclude that
$$
\P(\BB) \le \P(\tilde{\BB}) / \P(\GG) \le 2 \exp \Big( -n + \frac{n}{4 \ln \e^{-1}} \Big) 
\le \exp(-n/2).
$$
We have shown \eqref{eq: prob BB} and thus have completed the proof of the theorem. 
\end{proof}

\section{From $\infty \to 2$ norm to the operator norm: controlling the bounded entries}		\label{s: bounded}

In Theorem~\ref{thm: infty to 2}, we gave an optimal $O(n)$ bound for the $\infty \to 2$ norm of a 
random matrix with few removed columns. We will now convert this into an optimal $O(\sqrt{n})$ bound
for the operator norm. This can be done by applying a form of Grothendieck-Pietsch theorem 
(see \cite[Proposition 15.11]{LT}), which has been used recently in \cite[section 3.2]{LV} in a similar context. 

\begin{theorem}[Grothendieck-Pietsch]\label{G-P}
Let $B$ be a $k \times m$ real matrix and $\delta > 0$. 
Then there exists $J \subset [m]$ with $|J| \le \delta m$ such that
$$
\|B_{J^c} \| \le \frac{2 \|B\|_{\infty\to 2}}{\sqrt{\delta m}}.
$$
\end{theorem} 

Applying Theorem~\ref{thm: infty to 2} followed by Grothendieck-Pietsch theorem, 
we obtain the following result. 

\begin{lemma}[Bounding the operator norm by removing a few columns]				\label{lem: op columns}
  Consider an $n \times n$ random matrix $A$ with i.i.d. entries $A_{ij}$ which have mean zero and at most unit variance 
  and satisfy \eqref{trunc_probab_model}. Let $\e \in (0,1]$. 
  Then with probability at least $1 - 2\exp(-\e n/2)$, there exists a subset $J \in [n]$ with cardinality $|J| \le \e n$ 
  such that 
  $$
  \|A_{J^c}\| \le C \sqrt{\frac{\ln \e^{-1}}{\e}} \cdot \sqrt{n}.
  $$
\end{lemma}	
 
\begin{proof} 
Apply Theorem~\ref{thm: infty to 2} for $\e/2$ instead of $\e$. 
We obtain a subset of columns $J_1 \subset [n]$, $|J_1| \le \e n/2$, which satistfies 
\begin{equation}         \label{eq: bounding infty to 2}
\|A_{J_1^c}\|_{\infty \to 2} \le C \sqrt{\ln \e^{-1}} \cdot n
\end{equation}
with probability at least $1 - 2\exp(-\e n/2)$.

Next apply Grothendieck-Pietsch Theorem~\ref{G-P} for the matrix $A_{J_1^c}$ and for $\delta = \e/2$. 
We obtain a further subset $J_2 \subset J_1^c$, $|J_2| \le \d |J_1^c| \le \e n/2$, such that 
the removal of columns in both $J := J_1 \cup J_2$ leads to 
$$
\|A _{J^c}\| 
\le \frac{2 \|A _{J_1^c}\|_{\infty\to 2}}{\sqrt{\d |J_1^c|}} 
\lesssim C \sqrt{\frac{\ln \e^{-1}}{\e}} \cdot \sqrt{n}.
$$
In the last inequality, we used the bound \eqref{eq: bounding infty to 2} and that 
$\d = \e/2$ and $|J_1^c| \ge n - \e n/2 \ge n/2$.
The proof is complete.
\end{proof}

We are ready to prove a partial case of Theorem~\ref{main}, for the matrices whose entries 
are $O(\sqrt{n})$. It follows by applying Lemma~\ref{lem: op columns} for $A$ and $A^\tran$ separately, 
and then superposing the results.

\begin{proposition}[Bounded entries]		\label{prop: bounded entries}
  Consider an $n \times n$ random matrix $A$ with i.i.d. entries $A_{ij}$ which have mean zero and at most unit variance 
  and satisfy \eqref{trunc_probab_model}. Let $\e \in (0,1]$.  Then with probability at least $1 - 4\exp(- \e n/2)$, 
  there exists an $\e n \times \e n$ submatrix of $A$ such that replacing all of its entries with zero 
  leads to a well-bounded matrix $\tilde{A}$: 
  $$
  \|\tilde{A}\| \leq C \sqrt{\frac{\ln \e^{-1}}{\e}} \cdot \sqrt{n}.
  $$
\end{proposition}

\begin{proof}
Apply Lemma~\ref{lem: op columns} for $A$ and $A^\tran$. We obtain that with probability at least 
$1 - 4\exp(-\e n/2)$, there exists sets $I$ and $J$ with at most $\e n$ indices in each, and such that 
\begin{equation}         \label{eq: A Atran}
\|A_{[n] \times J^c}\| \lesssim \sqrt{\frac{\ln \e^{-1}}{\e}} \cdot \sqrt{n} \quad \text{and} \quad
\|A_{I^c \times [n]}\| \lesssim \sqrt{\frac{\ln \e^{-1}}{\e}} \cdot \sqrt{n}.
\end{equation}
We claim that $\tilde{A} := A_{(I \times J)^c}$ satisfies the conclusion of the proposition.
The support of this matrix, $(I \times J)^c$, is a disjoint union of two sets, $[n] \times J^c$ and $I^c \times J$. 
Then, using the triangle inequality, we have
$$
\|A_{(I \times J)^c}\| \le \|A_{[n] \times J^c}\| + \|A_{I^c \times J}\|.
$$
We already controlled the first term in \eqref{eq: A Atran}. As for the second term, since 
adding columns can only increase the operator norm, we have $\|A_{I^c \times J}\| \le \|A_{I^c \times [n]}\|$, 
which we also bounded in \eqref{eq: A Atran}. The proof is complete.
\end{proof}

\section{Controlling the large entries, and completing the proof of Theorem~\ref{main}}		\label{s: large entries}

In the previous section, we proved a partial case of Theorem~\ref{main} that controls relatively 
small entries of $A$, those of the order $O(\sqrt{n})$. Larger entries will be controlled in this section. 

\subsection{Bernoulli random matrices and random graphs}

The following general lemma will help us analyze the patterns such large entries can form.

\begin{lemma}[Bernoulli random matrix]				\label{lem: Bernoulli}
  Let $B$ be an $n \times n$ random matrix whose entries are independent Bernoulli random 
  variables with mean $p$. Let $\e \in (0,1/2]$.
  Consider the rows of $B$ with more than $21pn + 2 \ln \e^{-1}$ ones. 
  Then with probability $1-\exp(-\e n/2)$, these rows have at most $\e n$ ones altogether.
\end{lemma}

To see the connection to our original problem, we will later choose the entries of $B$ to be 
the indicators of the large entries of $A$. 

\begin{proof}
Let $S_i$ denote the number of ones in the $i$-th row of $B$. 
Then $\E S_i = pn$. A standard application of Chernoff's inequality shows that
\begin{equation}         \label{eq: Si tails}
\Pr{ S_i > t } \le e^{-2t} \quad \text{for } t \ge 21pn.
\end{equation}
Let $K \ge 21pn$ be a number to be chosen later. (We will eventually choose $K$ as $21pn + 2\ln \e^{-1}$
as in the statement of the lemma.) Define the random variables
$$
X_i := S_i \ind_{\{S_i > K\}}.
$$
The quantity of interest is the total number of ones in the heavy rows, and it equals $\sum_{i=1}^n X_i$. 
To control this sum of independent random variables, we can use the standard 
Bernstein's trick (commonly called Chernoff's bound), where we use Markov's inequality 
after exponentiation. We obtain
\begin{equation}         \label{eq: sum Xi tail}
\Pr{ \sum_{i=1}^n X_i > \e n } 
\le e^{-\e n} \E  \exp \Big( \sum_{i=1}^n X_i \Big)
= \Big[ e^{-\e} \E e^{X_1} \Big]^n,
\end{equation}
where the last equality follows by independence and identical distribution. 
Now, by definition of $X_1$ we have
\begin{align*}
\E e^{X_1} 
&= \E e^{X_1} \ind_{\{X_1 = 0\}} + \E e^{X_1} \one_{\{X_1 \ne 0\}}
\le 1 + \E e^{S_1} \ind_{\{S_1 > K\}}\\
&= 1+ \int_{e^K}^\infty \Pr{ e^{S_1} > u } du \\
&= 1+ \int_{K}^\infty \Pr{ S_1 > t } e^t \, dt \quad \text{(by a change of variables)} \\
&\le 1 + \int_{K}^\infty e^{-2t} e^t \, dt \quad \text{(using \eqref{eq: Si tails} for $t \ge K \ge 21pn$)} \\
&= 1 + e^{-K} 
\le \exp(e^{-K}).
\end{align*}
Substituting this bound into \eqref{eq: sum Xi tail}, we conclude that 
$$
\Pr{ \sum_{i=1}^n X_i > \e n } 
\le \exp \left[ (-\e + e^{-K}) n \right] \le \exp(-\e n/2),
$$
if we choose $K$ so that $e^{-K} \le \e/2$. To finish the proof, recall that our argument works if 
$K$ satisfies the two conditions: $K \ge 21pn$ and $e^{-K} \le \e/2$. We thus choose 
$K := 21pn + 2\ln \e^{-1}$ and complete the proof. 
\end{proof}

\begin{corollary}[Bernoulli random matrix]				\label{cor: Bernoulli IJ removed}
  Let $B$ be an $n \times n$ random matrix whose entries are independent Bernoulli random 
  variables with mean $p$. Let $\e \in (0,1]$.
  Then with probability at least $1-2\exp(-\e n/4)$,
  there exists an $\e n \times \e n$ submatrix of $B$ such that replacing all of its entries with zero 
  leads to a matrix $\tilde{B}$ whose rows and columns have at most $21pn + 4 \ln \e^{-1}$ ones each.
\end{corollary}

\begin{proof}
Apply Lemma~\ref{lem: Bernoulli} for $B$ and $B^\tran$ with $\e/2$ instead of $\e$, 
and take the intersection of the two good events. With the required probability, 
we obtain a set of $\e n$ bad entries of $B$ whose removal makes 
all rows and columns of $B$ contain at most $21pn + 2 \ln \e^{-1}$ ones.
It remains to note that these $\e n$ entries can be trivially placed in some $\e n \times \e n$ submatrix of $B$, and deletion of the whole $\e n \times \e n$ submatrix can only decrease the number of non-zero elements in the rows and columns of the residual part.
\end{proof}

\begin{remark}[Random graphs]
  It is not difficult to obtain a version of Corollary~\ref{cor: Bernoulli IJ removed} for symmetric random matrices. 
  This version can be interpreted as a statement about Erd\"os-R\'enyi random graphs $G(n,p)$, 
  with $B$ playing the role of the adjacency matrix. It states
  that with high probability, one can make all degrees of a $G(n,p)$ random graph bounded by $O(pn + \ln \e^{-1})$
  after removing the internal edges from a sub-graph with $\e n$ vertices.
\end{remark}

\subsection{Moderately large entries}

We will use Corollary~\ref{cor: Bernoulli IJ removed} to deduce Theorem~\ref{main} for matrices with moderately
large entries. Namely, we assume here that all entries of $A$ satisfy 
\begin{equation}         \label{eq: moderately large}
A_{ij} = 0 \quad \text{or} \quad \frac{\sqrt{n}}{2} \le |A_{ij}| \le \frac{5\sqrt{n}}{\sqrt{\e}}.
\end{equation}

\begin{proposition}[Moderately large entries]		\label{prop: moderate entries}
  Consider an $n \times n$ random matrix $A$ with i.i.d. entries which satisfy $\E A_{ij}^2 \le 1$ and 
  \eqref{eq: moderately large}. Let $\e \in (0,1/2]$.
  Then with probability at least $1-2\exp(-\e n/4)$, 
  there exists an $\e n \times \e n$ submatrix of $A$ such that replacing all of its entries with zero 
  leads to a well-bounded matrix $\tilde{A}$: 
  \begin{equation}         \label{eq: bounded entries}
  \|\tilde{A}\| \leq \frac{C \ln \e^{-1}}{\sqrt{\e}} \cdot \sqrt{n}.
  \end{equation}
\end{proposition}

\begin{proof}
Consider the matrix $B$ whose elements are indicators of moderately large entries of $A$, 
i.e. 
$$
B_{ij} := \ind_{\{A_{ij} \ne 0\}}.
$$
Then $B_{ij}$ are i.i.d. Bernoulli random variables with mean 
\begin{equation}         \label{eq: p}
p := \E B_{ij} = \Pr{A_{ij} \ne 0} \le \Pr{|A_{ij}| \ge \frac{\sqrt{n}}{2} }
\le \frac{2}{n}.
\end{equation}
(In the last inequality, we used Chebyshev's inequality and the assumption $\E A_{ij}^2 \le 1$.)
Corollary~\ref{cor: Bernoulli IJ removed} applied to $B$ gives us an $\e n \times \e n$ submatrix of $A$ 
such that the number of non-zero elements in every row and column of $\tilde{A}$
(obtained by zeroing out the elements of $A$ outside that submatrix) is bounded by 
\begin{equation}         \label{eq: rows columns nonzero}
21pn + 4 \ln \e^{-1} \lesssim \ln \e^{-1},
\end{equation}
where we used \eqref{eq: p} in the last bound.

Moreover, assumption \eqref{eq: moderately large} shows that all entries of $\tilde{A}$ are bounded in absolute 
value by $5\sqrt{n} / \sqrt{\e}$. This and \eqref{eq: rows columns nonzero} imply that the $\ell_1$ norm of 
all rows $\tilde{A}_i$ and columns $\tilde{A}^j$ can be bounded as follows: 
$$
\max_{i,j} \left( \|\tilde{A}_i\|_1, \, \|\tilde{A}^j\|_1 \right) \lesssim \frac{\sqrt{n}}{\sqrt{\e}} \cdot \ln \e^{-1}.
$$
Applying Lemma~\ref{lem: norm L1 rows cols} leads to \eqref{eq: bounded entries}.
\end{proof}

\subsection{Very large entries}

Finally, we will need to prove Theorem~\ref{main} for very large entries -- now we assume that all entries of $A$ satisfy
\begin{equation}         \label{eq: very large}
A_{ij} = 0 \quad \text{or} \quad |A_{ij}| > \frac{5\sqrt{n}}{\sqrt{\e}}.
\end{equation}
There are typically very few such entries, as the following simple result shows. 

\begin{lemma}[Few very large entries]			\label{lem: very large}
  Consider an $n \times n$ random matrix $A$ with i.i.d. entries which satisfy $\E A_{ij}^2 \le 1$ and 
  \eqref{eq: very large}. Let $\e \in (0,1/2]$.
  Then with probability at least $1-\exp(-\e n)$, 
  the matrix $A$ has at most $\e n$ non-zero entries. 
\end{lemma}

\begin{proof}
Using Chebyshev's inequality and the assumption that $\E A_{ij}^2 \le 1$, we see that 
the probability that a given entry is nonzero is
$$
\Pr{A_{ij} \ne 0} \le \Pr{|A_{ij}| > \frac{5\sqrt{n}}{\sqrt{\e}} } \le \frac{\e}{25 n}.
$$
Thus the expected number of non-zero entries in $A$ is at most $\e n / 25$.
A standard application of Chernoff's inequality (see e.g. \cite[Chapter~2]{V HDP}) gives 
$$
\Pr{ \text{$A$ has more than $\e n$ nonzero entries} } \le e^{-\e n}.
$$
The proof is complete.
\end{proof}

Since a set of $\e n$ indices can be always placed in an $\e n \times \e n$ submatrix,
we can state Lemma~\ref{lem: very large} as follows.

\begin{corollary}[Very large entries]			\label{cor: very large entries}
  Consider an $n \times n$ random matrix $A$ with i.i.d. entries which satisfy $\E A_{ij}^2 \le 1$ and 
  \eqref{eq: very large}. Let $\e \in (0,1/2]$.
  Then with probability at least $1-\exp(-\e n)$, all non-zero entries of $A$ are contained 
  in an $\e n \times \e n$ submatrix.
\end{corollary}

\subsection{Proof of Theorem~\ref{main}}			\label{s: proof main}

We are going to assemble Proposition~\ref{prop: bounded entries} for the bounded entries of $A$, Proposition~\ref{prop: moderate entries} for moderately large entries, 
and Corollary~\ref{cor: very large entries} for very large entries.
The $\e n \times \e n$ sub-matrices that appear in these results are possibly different.
The following simple lemma will help us to combine them into one.

\begin{lemma}			\label{lem: submatrix deletion}
  Let $B$ be a matrix. Zeroing out any submatrix of $B$ cannot increase the operator norm more than twice.
\end{lemma}

\begin{proof}
Let $\tilde{B}$ denotes the matrix obtained by zeroing out a submatrix of $B$ spanned by the index set $I \times J$. Triangle inequality gives
$$
 \|\tilde{B}\| \le \|B_{I \times J^c}\| + \|B_{I^c \times [n]}\| \le \|B\| + \|B\|.
$$
The last inequality follows from the fact that zeroing out any subset of rows or columns cannot increase the operator norm. 
\end{proof}

\begin{remark}
  One may wonder if zeroing out a submatrix can increase the norm at all.
  Basic simulations show that it actually can.
\end{remark}

\begin{proof}[Proof of Theorem~\ref{main}]
Decompose $A$ into a sum of three $n \times n$ matrices with disjoint support,
\begin{equation}\label{eq: decomp1}
A = B + M + L,
\end{equation}
where $B$ contains bounded entries of $A$ -- those that satisfy $|A_{ij}| \le \sqrt{n}/2$,
the matrix $M$ contains moderately large entries -- those for which $\sqrt{n}/2 < |A_{ij}| \le 5\sqrt{n/\e}$,
and $L$ contains large entries -- those satisfying $|A_{ij}| > 5\sqrt{n/\e}$.

To bound $B$, let us subtract the mean and first bound 
$$
G := B - \E B.
$$
The entries of this matrix have zero mean and satisfy
$$
\E G_{ij}^2 = \Var(B_{ij}) \le \E B_{ij}^2 \le \E A_{ij}^2 = 1 
$$
(where we used the moment assumption) and 
$$
\|G_{ij}\|_\infty = \|B_{ij} - \E B_{ij}\|_\infty
\le 2 \|B_{ij}\|_\infty \le \sqrt{n}.
$$
Thus we can apply Proposition~\ref{prop: bounded entries} for $0.5 G$. It says that 
with probability at least $1 - 4\exp(- \e n/2)$, the removal of a certain 
$\e n \times \e n$ submatrix of $G$ leads to a well-bounded matrix $\tilde{G}$, i.e. 
$$
\|\tilde{G}\| \lesssim \sqrt{\frac{\ln \e^{-1}}{\e}} \cdot \sqrt{n}.
$$

Next, we bound $\E B$, a matrix whose entries are the same. Thus
\begin{align*}
\|\E B\| 
  &= n \big|\E B_{ij}\big|
  = n \big|\E (A_{ij} - B_{ij})\big|		\quad \text{(since $\E A_{ij} = 0$)} \nonumber\\
  &= n \big|\E A_{ij} \ind_{|A_{ij}|>\sqrt{n}/2} \big|	\quad \text{(by definition of $A_1$)} \nonumber\\
  &\le n \big( \E A_{ij}^2 \big)^{1/2} \big( \Pr{|A_{ij}|>\sqrt{n}/2} \big)^{1/2} 
  	\quad \text{(by Cauchy-Schwarz)} \nonumber\\
 &\le n \cdot 1 \cdot 2/\sqrt{n} \le 2\sqrt{n}  \quad \text{(by Chebyshev's inequality)}.
\end{align*}

To bound $M$ and $L$, note that 
$$
\E M^2_{ij} \le \E A_{ij}^2 = 1 \text{ and } \E L^2_{ij} \le \E A_{ij}^2 = 1.
$$
Thus, Proposition~\ref{prop: moderate entries} can be applied to $M$: with probability $1-\exp(-\e n/4)$ there exist a $\e n \times \e n$ submatrix of $M$, the removal of which leads to a matrix $\tilde{M}$, such that
$$
  \|\tilde{M}\| \leq \frac{C \ln \e^{-1}}{\sqrt{\e}} \cdot \sqrt{n}.
$$
Proposition~\ref{cor: very large entries} can be applied to $L$. We can interpret its conclusion as follows: with probability $1-2\exp(-\e n)$ there exist a $\e n \times \e n$ submatrix of $L$, the removal of which leads to an identically zero matrix $\tilde{L}$, i.e.
$$\|\tilde{L}\| = 0.$$

Now, let us embed the three $\e n \times \e n$ submatrices of $A$ we just constructed into one $3\e n \times 3\e n$ submatrix and zero out this whole submatrix. By Lemma~\ref{lem: submatrix deletion}, the norms of $\tilde{G}$, $\E B$, $\tilde{M}$ and $\tilde{L}$ will not increase more than twice as a result of this operation. Taking intersection of the three good events, we conclude that with the probability at least
$$
1 - 4\exp(- \e n/2) - 2\exp(-\e n/4) - \exp(-\e n)
 \ge 1 - 7 \exp(-\e n /4)
$$
there exists an $3 \e n \times 3 \e n$ submatrix of $A$ such that replacing all of its entries by zero 
leads to a well-bounded matrix $\tilde{A}$:
$$
\|\tilde{A}\| \lesssim \sqrt{\frac{\ln \e^{-1}}{\e}} \cdot \sqrt{n} 
	+ 2\sqrt{n} + \frac{\ln \e^{-1}}{\sqrt{\e}} \cdot \sqrt{n} + 0
\lesssim \frac{\ln \e^{-1}}{\sqrt{\e}} \cdot \sqrt{n}. 
$$
This proves the conclusion of Theorem~\ref{main} 
with $3\e$ instead of $\e$, where $\e \in (0, 1/2]$ is arbitrary. 
By rescaling, Theorem~\ref{main} holds also as originally stated. This concludes the proof of Theorem~\ref{main}.
\end{proof}

\section{Global problem: proof of Theorem~\ref{global}} \label{s: global}

In this section we prove Theorem~\ref{global}, which states that either nonzero mean or infinite second moment 
make it impossible to repair the matrix norm by removing a small submatrix. 
We will first prove a non-asymptotic version version of this result. Once this is done,
an application of Borel-Cantelli Lemma will quickly yield Theorem~\ref{global}.

\begin{proposition}[Global problem: non-asymptotic regime] \label{prop: global}
  Consider an $n \times n$ random matrix $A$ whose entries are i.i.d.   
  random variables that have either nonzero mean or infinite second moment, and let $\e \in (0,1)$.   
  Then, for any $M > 0$ there exists $n_0$ that may depend only on $\e$, $M$ and the distribution of the entries, 
  and such that for any $n > n_0$ 
  the following event holds with probability at least $1-e^{-n}$:
  every $(1-\e)n \times (1-\e)n$ submatrix $A'$ of $A$ satisfies
  $$
  \|A'\| \ge M\sqrt{n}.
  $$
\end{proposition}

Before we prove this proposition, let us pause to see its connection to the matrix $\tilde{A}_n$ of Theorem~\ref{global}. 
Proposition~\ref{prop: global} yields that this matrix satisfies
$$
\|\tilde{A}_n\| \ge M \sqrt{n}.
$$
Indeed, modifying an $\e n \times \e n$ submatrix always leaves some $(1-\e) n \times (1-\e) n$ submatrix $A'$ intact, 
so we can apply Proposition~\ref{prop: global} for that submatrix.

\subsection{Infinite second moment} \label{s: inf var}

Here we will prove the part of Proposition~\ref{prop: global} about infinite second moment; 
the case of nonzero mean will be treated in Section~\ref{s: nonzero mean}. 
Let us start with the following lemma which will help us treat a fixed submatrix.

\begin{lemma}			\label{lem: fixed sub-matrix}
  Consider an $m \times m$ random matrix $B$ whose entries are i.i.d.
  random variables with infinite second moment.
  Then, for any $M > 0$ there exists $m_0$ that may depend only on $M$ and the distribution of the entries, 
  and such that for any $m > m_0$ we have
  $$
  \|B\| \ge M\sqrt{m}
  $$
  with probability at least $1-\exp(-M^2 m)$.
\end{lemma}

\begin{proof}
By assumption, we have $\E B_{ij}^2 = \infty$. Therefore, for any $M > 0$ one can find a truncation level $K$ 
that depends only on $M$ and the distribution, and such that the truncated random variables
\begin{equation}         \label{eq: Bij truncated}
\bar{B}_{ij} := B_{ij} \ind_{|B_{ij}| \le K}
\quad \text{satisfy} \quad
\E \bar{B}_{ij}^2 \ge 2 M^2.
\end{equation}
(This follows easily from Lebesgue's monotone convergence theorem.)

Consider the matrix $\bar{B}$ with entries $\bar{B}_{ij}$. We have
$$
\|B\| \ge \frac{1}{\sqrt{m}} \|B\|_F \ge \frac{1}{\sqrt{m}} \|\bar{B}\|_F.
$$
Then we bound the failure probability as follows:
\begin{align*}
\Pr{\|B\| < M \sqrt{m}} 
  &\le \Pr{\|\bar{B}\|_F < Mm} 
  = \Pr{ \sum_{i,j=1}^m \bar{B}_{ij}^2 < M^2 m^2 } \\
  &\le \Pr{ \sum_{i,j=1}^m (\bar{B}_{ij}^2 - \E \bar{B}_{ij}^2) < -M^2 m^2 } 
\end{align*}
where we used \eqref{eq: Bij truncated} in the last step.

Apply Hoeffding's inequality for the random variables $\bar{B}_{ij}^2$ and use that they are
bounded by $K^2$ by construction. The probability above gets bounded by 
$$
\exp \Big( -\frac{M^4 m^2}{2K^2} \Big).
$$
If $m > 2K^2/M^2 = m_0$, this probability can be further bounded by $\exp(-M^2 m)$,
as claimed. 
\end{proof}

\smallskip

\begin{proof}[Proof of Proposition~\ref{prop: global} for infinite second moment]
We can assume without loss of generality that $M$ is large enough depending on $\e$. 
(Indeed, once the conclusion of the proposition holds for one value of $M$ it automatically holds for all smaller values.)

Apply Lemma~\ref{lem: fixed sub-matrix} for an $m \times m$ matrix $A'_n$ with $m = (1-\e)n$, 
and then take a union bound over all $\binom{n}{m}^2$ possible choices of such submatrices. 
It follows that the conclusion of Proposition~\ref{prop: global} holds with probability at least
$$
1 - \binom{n}{m}^2 \exp(-M^2 m).
$$
By Stirling's approximation, we have $\binom{n}{m} \le (en/m)^m$. Using this and substituting
$m=(1-\e)n$, we bound the probability below by 
$$
1 - \exp \Big[ \Big( 2\log \frac{e}{1-\e} - M^2 \Big) (1-\e) n \Big].
$$
If the value of $M$ is sufficiently large depending on $\e$, this probability is larger than
$1-\exp(-n)$, as claimed. Proposition~\ref{prop: global} for infinite second moment is proved.
\end{proof}

\subsection{Nonzero mean}  \label{s: nonzero mean}

Now we will prove the part of Proposition~\ref{prop: global} about nonzero mean. 
We can assume here that the second moment of the entries $A_{ij}$ is finite, 
as the opposite case was treated in Section~\ref{s: inf var}.
As before, we will first focus on one submatrix. 
In the following lemma we make an extra boundedness assumption, 
which we will get rid of using truncation later. 

\begin{lemma}		\label{lem: under boundedness}
  Consider an $m \times m$ random matrix $B$ whose entries are i.i.d. random variables 
  that satisfy
  $$
  \E B_{ij} = \mu > 0, \quad \E B_{ij}^2 \le \s^2, \quad |B_{ij}| \le K \sqrt{m} \text{ a.s.}
  $$
  Then, for any $M > 0$ there exists $m_0$ that may depend only on $\mu$, $\sigma$, $K$ and $M$,
  and such that for any $m > m_0$ we have
  $$
  \|B\| \ge \frac{\mu m}{2}
  $$
  with probability at least $1-\exp(-M^2 m)$.
\end{lemma}

\begin{proof}
Notice that
$$
\|B\| \ge \frac{1}{m} \sum_{i,j=1}^m B_{ij}.
$$
(To check this inequality, recall that $\|B\| \ge x^\tran B x$ for any unit vector $x$; use this for the vector $x$
whose all coordinates equal $1/\sqrt{m}$.)
Then we can bound the failure probability as follows:
$$
\Pr{ \|B\| < \frac{\mu m}{2} } 
  \le \Pr{ \sum_{i,j=1}^m B_{ij} < \frac{\mu m^2}{2} } 
  \le \Pr{ \sum_{i,j=1}^m (B_{ij} - \E B_{ij}) < - \frac{\mu m^2}{2} } 
$$
where we used that $\E B_{ij} = \mu$ in the last step.

Apply Bernstein's inequality for the random variables $B_{ij}$ and use that they 
have variance at most $\s^2$ and are bounded by $K\sqrt{m}$ by assumption. 
The failure probability gets bounded by 
$$
\exp \Big( - \frac{\mu^2 m^4 / 8}{\s^2 m^2 + K\sqrt{m}/3} \Big).
$$
If $m$ is large enough depending $\mu$, $\s$, $K$ and $M$, then this probability 
can be further bounded by $\exp(-M^2 m)$, as claimed. 
\end{proof}

Next, we will use truncation to get rid of the boundedness assumption in Lemma~\ref{lem: under boundedness}
and thus prove the following. 

\begin{lemma}		\label{lem: without boundedness}
  Consider an $m \times m$ random matrix $B$ whose entries are i.i.d. random variables 
  that satisfy
  $$
  \E B_{ij} = \mu > 0, \quad \E B_{ij}^2 \le \s^2.
  $$
  Then, for any $M > 0$ there exists $m_0$ that may depend only on $\mu$, $\sigma$, $K$, $M$
  and the distribution of the entries, 
  and such that for any $m > m_0$ we have
  \begin{equation}         \label{eq: without boundedness}
  \|B\| \ge M \sqrt{m}
  \end{equation}
  with probability at least $1-\exp(-M^2 m)$.
\end{lemma}

\begin{proof}
Choosing $m_0$ large enough depending on $M$ and the distribution of $B_{ij}$, 
we can make sure that for any $m \ge m_0$ the truncated random variables
$$
\bar{B}_{ij} := B_{ij} \ind_{|B_{ij}| \le M \sqrt{m}}
\quad \text{satisfy} \quad
\E \bar{B}_{ij} \ge \E B_{ij} - \frac{\mu}{2} = \frac{\mu}{2}.
$$
(This follows easily from Lebesgue's monotone convergence theorem.)

Let us consider the event that all entries of $B$ are appropriately bounded: 
$$
\EE := \left\{ |B_{ij}| \le M \sqrt{m} \text{ for all } i,j \in [n] \right\}.
$$
Suppose for a moment that \eqref{eq: without boundedness} fails, so we have $\|B\| < M \sqrt{m}$.
Since the inequality $\|B\| \ge \max_{i,j} |B_{ij}|$ is always true, the event $\EE$ must hold in this case.
This in turn implies that the truncation has no effect on the entries, i.e. $\bar{B}_{ij} = B_{ij}$ for all $i,j$.

We have shown that in the event of the failure of \eqref{eq: without boundedness}, 
we may automatically assume that the entries of $B$ are appropriately bounded.
Therefore the failure probability satisfies
$$
\Pr{\|B\| < M \sqrt{m}} = \Pr{\|\bar{B}\| < M \sqrt{m}}
$$
where $\bar{B}$ denotes the matrix with the truncated entries $\bar{B}_{ij}$.
It remains to apply Lemma~\ref{lem: under boundedness} for the random matrix $\bar{B}$, 
noting that truncation may only decrease the second moment. 
The failure probability gets bounded by $\exp(-M^2 m)$, as claimed. 
\end{proof}

\smallskip

\begin{proof}[Proof of Proposition~\ref{prop: global} for non-zero mean]
As we mentioned in the beginning of this section, we can assume that the entries $B_{ij}$
have finite second moment $\s^2$. Then the conclusion of the proposition 
follows by exact same union bound argument as in the end of Section~\ref{s: inf var}
(just use Lemma~\ref{lem: without boundedness} instead of Lemma~\ref{lem: fixed sub-matrix} there.)
\end{proof}

\subsection{Proof of Theorem~\ref{global}}

We will prove a stronger fact that
\begin{equation}         \label{eq: A'n}
\min \frac{\|A'_n\|}{\sqrt{n}} \to \infty \quad \text{as} \quad n \to \infty \quad \text{almost surely},
\end{equation}
where the minimum is taken over all $(1-\e) n \times (1-\e) n$ submatrices $A'_n$ of $A_n$. 
As we mentioned below Proposition~\ref{prop: global}, this would imply the conclusion of Theorem~\ref{global},
since modifying an $\e n \times \e n$ submatrix leaves some $(1-\e) n \times (1-\e) n$ sub-matrix intact.

Fix any $M > 0$ and consider the events
$$
\EE_n := \Big\{ \min \frac{\|A'_n\|}{\sqrt{n}} \ge M \Big\}, \quad n=1,2,\ldots
$$ 
where the minimum has the same meaning as before.
By Proposition~\ref{prop: global}, there exists $n_0$ such that
$$
\P(\EE_n^c) \le e^{-n} \text{ for all } n > n_0.
$$
In particular, the series $\sum_{n=1}^\infty \P(\EE_n^c)$ converges. Borel-Cantelli lemma then implies that
the probability that infinitely many $\EE_n^c$ occur is $0$. Equivalently, with probability $1$
there exists $N$ such that $\EE_n$ hold for all $n \ge N$. 

We have shown that for any $M>0$, with probability $1$ there exists $N$ such that 
$$
\min \frac{\|A'_n\|}{\sqrt{n}} \ge M \quad \text{for all } n \ge N.
$$
Intersecting these almost sure events for $M=1,2,\ldots$, we conclude \eqref{eq: A'n}. 
Theorem~\ref{global} is proved. \qed

\section{Removal of large entries under $2+\e$ moments} \label{s: twopluseps}

In this section we give a proof of Proposition~\ref{twoplus} based on a general bound on random matrices
due to A.~Bandeira and R.~van Handel \cite{BvH}.

\subsection{Proof of Proposition~\ref{twoplus}} 
Let us call an entry $A_{ij}$ {\em large} if $|A_{ij}| > R := n^{1/2 - \e/8}$, 
otherwise call the entry {\em small}. 

We claim that there are very few large entries with high probability, and we can check this 
by the same argument as in Lemma~\ref{lem: very large}. 
Indeed, the $2+\e$ moment assumption and Chebyshev's inequality give
\begin{equation}		\label{eq: entry large prob}
\Pr{A_{ij} \text{ is large}} = \Pr{|A_{ij}|>R} < \frac{1}{R^{2+\e}} \le \frac{1}{n^{1+\e/8}},
\end{equation}
where the last inequality follows by our choice of $R$.
Thus the expected number of large entries is at most $n^2/n^{1+\e/8} = n^{1-\e/8}$.
A standard application of Chernoff's inequality (see e.g. \cite[Chapter~2]{V HDP}) gives 
$$
\Pr{ \text{$A$ has more than $n^{1-\e/9}$ large entries} } 
\le \Big( \frac{e n^{1-\e/8}}{n^{1-\e/9}} \Big)^{n^{1-\e/9}},
$$
which can be further bounded by $\exp(-n^{1/2})$ if $n$ is sufficiently large in terms of $\e$.
Hence we can zero out all large large entries of $A$. It remains to show that
the result of this operation, which we denote by $\tilde{A}$,
has norm at most $6\sqrt{n}$ with high probability.

For convenience, let us subtract the mean, and first bound 
$$
G := \tilde{A} - \E \tilde{A},
$$
which is an $n \times n$ random matrix with independent mean zero entries $G_{ij}$. 
A theorem of A.~Bandeira and R.~van Handel (see \cite[Remark~3.13]{BvH}) states that for any $t > 0$, 
we have
\begin{equation}\label{bvh}
  \P\{\|G\| \ge 4\sigma + t \} \le n \exp( - c t^2 / \sigma_*^2),
\end{equation}
where $c>0$ is an absolute constant and
$$ 
\sigma^2 := \max_i \sum_j \E (G_{ij}^2), \qquad \sigma_* := \max_{ij} \|G_{ij}\|_{\infty}.
$$
In our case, 
$$
\E G_{ij}^2 = \Var(\tilde{A}_{ij}) \le \E \tilde{A}_{ij}^2 \le \E A_{ij}^2 \le 1
$$
(where we used the moment assumption) and
$$
\|G_{ij}\|_\infty = \|\tilde{A}_{ij} - \E \tilde{A}_{ij}\|_\infty \le 2 \|\tilde{A}_{ij}\|_\infty \le 2R.
$$
Hence 
$$
\s \le \sqrt{n} \quad \text{and} \quad \s_* \le 2R.
$$
Then, using \eqref{bvh} with $t=\sqrt{n}$, we conclude that
\begin{equation}		\label{eq: G norm tail}
\P\{\|G\| \ge 5\sqrt{n} \} \le n \exp( - cn/4R^2) 
\le \exp(-n^{\e/5})
\end{equation}
where the last inequality holds due to the definition of $R$, if $n$ is sufficiently large in terms of $\e$.

Finally, we need to bound the contribution of the mean $\E \tilde{A}$ which we subtracted in defining $G$.
This can be done by the exactly same argument as we used in the proof of Theorem~\ref{main} in 
Section~\ref{s: proof main}. We repeat it here for completeness. Note that all entries of $\E \tilde{A}$ are the same, thus
\begin{align}		\label{eq: E tilde A norm}
\|\E \tilde{A}\| 
  &= n \big|\E \tilde{A}_{ij}\big|
  = n \big|\E (A_{ij} - \tilde{A}_{ij})\big|		\quad \text{(since $\E A_{ij} = 0$)} \nonumber\\
  &= n \big|\E A_{ij} \ind_{|A_{ij}|>R} \big|	\quad \text{(by definition of $\tilde{A}$)} \nonumber\\
  &\le n \big( \E A_{ij}^2 \big)^{1/2} \big( \Pr{|A_{ij}|>R} \big)^{1/2} 
  	\quad \text{(by Cauchy-Schwarz)} \nonumber\\
 &\le n \cdot 1 \cdot n^{-1/2} \le n^{1/2},
\end{align}
where we used the moment assumption and a weaker form of the bound \eqref{eq: entry large prob}.

Concluding, it follows from \eqref{eq: G norm tail} and \eqref{eq: E tilde A norm}
that with probability at least $\exp(-n^{\e/5})$, we have
$$
\|\tilde{A}\| \le \|\tilde{A} - \E \tilde{A}\| + \|\E \tilde{A}\| \le 6\sqrt{n}.
$$
The proof of Proposition~\ref{twoplus} is complete.
\qed

\section{Further questions}		\label{s: questions}

Several extensions of Theorem~\ref{main} seem plausible. 

\begin{enumerate}[1.]

\item It is natural to expect a version Theorem~\ref{main} even if the entries of $A$ are {\em not identically distributed}. 
Our argument relies on the identical distribution in several places, including discretization arguments 
(proof of Theorem~\ref{thm: damping sum}) and symmetrization (proofs of Lemmas~\ref{lem: infty to 2 by permutations} 
and \ref{lem: infty to 2 by permutations high prob}). 

\item A version of Theorem~\ref{main} should hold for {\em symmetric matrices} $A$ with independent entries
on and above the diagonal. A simplest way to get this result would be to use Theorem~\ref{main}
to control the parts of $A$ above and below the diagonal separately, and then combine them. However, 
for this argument one would need a version of Theorem~\ref{main} for non-identical distributed entries.

\item Unlike Feige-Ofek's result \cite{FO} mentioned in Section~\ref{s: related}, Theorem~\ref{main} 
does not indicate what sub-matrix should be removed to improve the norm; it is rather an existential result. 
It would be nice to have an {\em explicit description of a submatrix to be removed}.

\item It would be good to {\em remove the logarithmic factor} $\ln \e^{-1}$ from the bound in Theorem~\ref{main}, 
or to show that this factor is necessary. Such bound would be optimal up to an absolute constant factor.

\item Finally, while Remark~\ref{rem: optimality} states
that {\em the dependence on $\e$} in Theorem~\ref{main} is optimal in general,
this dependence might be dramatically improved under a natural boundedness assumption. Namely, suppose that
the entries of $A$ are $O(\sqrt{n})$ almost surely. (In fact, most of the proof -- until Section~\ref{s: large entries} --
was done under this additional assumption.) In this case, is the dependence of the norm on $\e$ {\em logarithmic} in 
Theorem~\ref{main}, i.e. 
\begin{equation}\label{eq: bernoulli_bound_last}
\|\tilde{A}\| \leq C \ln(\e^{-1})\sqrt{n}?
\end{equation}
In fact, for the partial case of Bernoulli matrices such that $np = c_0 = const$ (where $p$ is a probability of a non-zero entry) this bound can be quickly deduced from Corollary \ref{cor: Bernoulli IJ removed}. 
 
Indeed, after renormalization that imposes matrix elements to have variance one (so we deal with the scaled Bernoulli matrix with $B_{ij} = O(p^{-1/2})$), we can see that such matrices satisfy the boundedness assumption, as $B_{ij} = O(p^{-1/2}) = O(\sqrt{n}/\sqrt{c_0}) = O(\sqrt{n})$.
 Then, by Corollary \ref{cor: Bernoulli IJ removed} after a deletion of $\e n \times \e n$ submatrix we get a matrix $\tilde{B}$ with all rows and columns having at most 
 $$21pn + 4 \ln\e^{-1} \le 100c_0\ln\e^{-1}$$
  non-zero elements of order $O(\sqrt{n})$. Hence, 
 $$
\max_{i,j} \left( \|\tilde{B}_i\|_1, \, \|\tilde{B}^j\|_1 \right) \lesssim \sqrt{n}\cdot \ln \e^{-1}.
$$
Applying Lemma~\ref{lem: norm L1 rows cols} leads to \eqref{eq: bernoulli_bound_last}.

\end{enumerate}

\end{document}